\documentclass[11pt,reqno,tbtags]{amsart}

\usepackage[utf8]{inputenc}

\usepackage{amsthm,amssymb,babel,amscd,amsthm,bbm,mathtools,appendix, xcolor, upgreek}
\usepackage{microtype}

\newcommand{\R}{\mathbb{R}}

\newcommand{\N}{\mathbb{N}}

\newcommand{\La}{\mathcal{L}_{\mathfrak{b},  \delta}}

\DeclareMathOperator{\supp}{supp}

\usepackage{tikz}
\usetikzlibrary{shapes.misc}
\tikzset{cross/.style={cross out}, minimum size=1pt, draw=black, inner sep =0pt, outer sep=0pt, cross/.default={1pt}}
\usepackage{xcolor, todonotes}
\definecolor{vertfonce}{rgb}{0.20, 0.46, 0.25}
\definecolor{rougefonce}{rgb}{0.64, 0.09, 0.20}
\usepackage{pgf,tikz,pgfplots}
\pgfplotsset{compat=1.15}
\usetikzlibrary{arrows}
\definecolor{zzttqq}{rgb}{0.6,0.2,0.}
\definecolor{xdxdff}{rgb}{0.49019607843137253,0.49019607843137253,1.}
\definecolor{uuuuuu}{rgb}{0.26666666666666666,0.26666666666666666,0.26666666666666666}
\definecolor{bblue}{rgb}{0,0.2,0.7}

\numberwithin{equation}{section}
\theoremstyle{definition}
\newtheorem{theorem}{Theorem}[section]
\newtheorem{lemma}[theorem]{Lemma}
\newtheorem{defin}[theorem]{Definition}
\newtheorem{prop}[theorem]{Proposition}

\newtheorem{remark}[theorem]{Remark}
\DeclarePairedDelimiter{\bigabs}{\bigg|}{\bigg|}
\DeclarePairedDelimiter{\abs}{\lvert}{\rvert}
\DeclarePairedDelimiter{\paren}{\lparen}{\rparen}

\DeclarePairedDelimiter{\norm}{\lVert}{\rVert}

\title[Magnetic Laplacian on almost flat magnetic barriers]{Discrete spectrum of the magnetic Laplacian on almost flat magnetic barriers.}
\author{Germ\'an Miranda}
\begin{document}
	
	\begin{abstract}
		The magnetic Laplacian with a step magnetic field has been intensively studied during the last years. We adapt the construction in~\cite{Almostflat} to prove the existence of bound states of a new effective operator involving a magnetic step field on a domain with an almost flat magnetic barrier. This result emphasizes the fact that even a small non-smoothness of the discontinuity region can cause the appearance of eigenvalues below the essential spectrum. We also give an example where this effective operator arises. 
	\end{abstract}
	\maketitle
	\section{Introduction}\label{Introduction section}
	\subsection{Motivation} The  Ginzburg--Landau model when the applied magnetic field is uniform has been intensively studied in several previous works. A good summary of these results can be found in the monograph \cite{FournaisHelfferbook}. The appearance of fabrication techniques allowing the creation of nonhomogeneous magnetic fields \cite{FodenLeadbeaterBurroughes, Smith1994} drew the attention to non-uniform fields. Later, piecewise constant magnetic fields were introduced in order to create graphene magnetic wave\-guides~\cite{Ghosh, Oroszlany}. These types of magnetic fields are of great interest, because they induce so-called snake states, which are trajectories of charge carriers propagating along the magnetic barrier \cite{ PeetersMatulis, Reijniers_2000, DombrowskHislop, HislopSoccorsi, HislopPopoff}. 

	In the last decade, models with piecewise constant magnetic fields, also called magnetic steps, have been further studied (see \cite{AssaadGiacomelli, AssaadBreakdown, AssadKachmarSundqvist, AssadInfluence}). In these works, smooth magnetic barriers in bounded domains are considered. In this paper, we present a new effective operator involving a magnetic step field where the magnetic barrier is not a smooth curve.
	
	\subsection{Main result}
	To keep the notation concise, we follow the one introduced in \cite{HislopPopoff}. Let \(\mathfrak{b}=(b_1, b_2)\in \R^2\) and \(\delta\in [0, \frac{\pi}{2})\). Moreover, consider the function \(g_{\delta}\) defined as
	\begin{equation}\label{Magnetic barrier line}
		g_{\delta}(x_1) = \begin{cases}
			x_1\tan{\delta} & \text{ if } x_1\leq 0\\
			0 & \text{ if } x_1>0.        
		\end{cases}
	\end{equation}
	Let
	\begin{equation}\label{Omega a}
		\Omega_{1, \delta}:= \{(x_1, x_2) \in \R^2 : x_2 <g_{\delta}(x_1)\},
	\end{equation}
	and
	\begin{equation}\label{Omega b}
		\Omega_{2,\delta}:= \{(x_1, x_2) \in \R^2 : x_2 >g_{\delta}(x_1)\}.
	\end{equation}
	{
		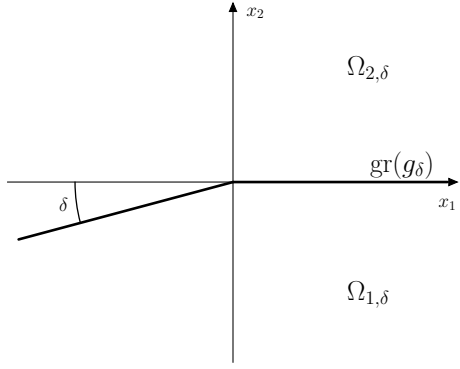
\begin{figure}[ht!]
			\scalebox{0.5}{
				\begin{tikzpicture}[line cap=round,line join=round,>=triangle 45,x=0.6cm,y=0.6cm]
					\clip(-10,-8) rectangle (10,8);
					\draw[thick,->] (-10,0) -- (10,0) node[anchor=north west]{}; 
					\draw[thick,->] (0,-8) -- (0,8) node[anchor=south east]{}; 
					\draw [line width=2.pt, domain=0:9.5] plot(\x,0); 
					\draw [line width=2.pt, domain=-9.5:0] plot(\x,{\x* 0.26794919243});
					\draw [line width=1.pt] (-7,0) arc (180: 195: 7); 
					\draw[color= black] (-7.5, -1) node {\Large \(\delta\)}; 
					\draw[color= black] (9.5, -1) node {\Large \(x_1\)}; 
					\draw[color= black] (1, 7.5) node {\Large \(x_2\)}; 
					\draw[color= black] (6, -5) node {\huge\(\Omega_{1, \delta}\)}; 
					\draw[color= black] (6, 5) node {\huge \(\Omega_{2, \delta}\)}; 
					\draw[color= black] (7.5, 0.75) node {\huge \(\text{gr} (g_{\delta})\)}; 
				\end{tikzpicture}
			}
			\caption{Graph of \(g_{\delta}\) and the two regions \(\Omega_{1, \delta}\) and \(\Omega_{2, \delta}\).}
		\end{figure}
		We consider the following operator in \(\R^2\)
		\begin{equation}\label{La}
			\mathcal{L}_{\mathfrak{b}, \delta}  = -(\nabla -i \mathcal{A}_{\mathfrak{b}, \delta})^2,
		\end{equation}
		with \(\mathcal{A}_{\mathfrak{b}, \delta} := (A_{\mathfrak{b}, \delta}, 0)\) 
		\begin{equation}\label{Aa,b,delta definition}
			A_{\mathfrak{b}, \delta} (x_1, x_2) = \begin{cases}
				- b_1 x_2 +\mathbbm{1}_{\R_-}(x_1)(b_1-b_2)x_1\tan{\delta} & \text{ if } (x_1, x_2) \in \Omega_{1, \delta}\\
				-b_2x_2 & \text{ if } (x_1, x_2) \in \overline{\Omega_{2,\delta}}
			\end{cases}
		\end{equation}
		for \(\delta \in (0, \frac{\pi}{2})\). Note that the magnetic potential \(\mathcal{A}_{\mathfrak{b},\delta}\in H^1_{\text{loc}} (\R^2, \R^2)\) and satisfies \[\text{curl }{\mathcal{A}}_{\mathfrak{b},\delta} = b_1 \mathbbm{1}_{\Omega_{1,\delta} }+ b_2\mathbbm{1}_{\Omega_{2,\delta}}.\]
		
		The domain of \(\mathcal{L}_{\mathfrak{b},\delta}\) is 
		\[\mathcal{D}(\mathcal{L}_{\mathfrak{b},\delta}) = \{u\in L^2(\R^2) : (\nabla -  i \mathcal{A}_{\mathfrak{b} , \delta }) u \in L^2(\R^2, \R^2) \text{ and }(\nabla -  i \mathcal{A}_{\mathfrak{b} , \delta })^2 u \in L^2(\R^2)\},\]
		and the associated quadratic form 
		\begin{equation}\label{Q_a,b,delta}
			\mathcal{Q}_{\mathfrak{b}, \delta} (u) = \int_{\R^2} |(\nabla - i\mathcal{A}_{\mathfrak{b},\delta}) u|^2 \, dx, 
		\end{equation}
		with \(\mathcal{D}(\mathcal{Q}_{\mathfrak{b},\delta}) = \{u\in L^2(\R^2) : (\nabla - i\mathcal{A}_{\mathfrak{b},\delta}) u\in L^2(\R^2, \R^2) \}\).  
		
		The operator \(\mathcal{L}_{\mathfrak{b},\delta}\) is self-adjoint; let \(\lambda_{\mathfrak{b}}(\delta) := \inf \text{Spec}\paren*{\mathcal{L}_{\mathfrak{b},\delta}}\). In Appendix~\ref{AppendixA} we prove that \(\inf\text{Spec}_{\text{ess}}(\mathcal{L}_{\mathfrak{b} , \delta}) = \beta_{\mathfrak{b}}\) for \(\mathfrak{b} = (1, b)\) or \(\mathfrak{b}=(b,1)\) with \(b\in [-1,0)\), where \(\beta_{\mathfrak{b}}\) is the bottom of the spectrum of \(\mathcal{L}_{\mathfrak{b},0}\).  
		
		Our main result says that if the smaller intensity of the field is in the bigger region \(\Omega_{2, \delta}\), then, for \(\delta\) small enough, there will be at least one eigenvalue below the essential spectrum of \(\mathcal{L}_{\mathfrak{b}, \delta}\). 
		\begin{theorem}\label{Main Theorem}
			If \(\mathfrak{b} = (1, b)\) with \(b\in (-1,0)\), then there exists \(\delta_0 \in (0, \frac{\pi}{2})\) such that, for all \(\delta\in (0, \delta_0)\), 
			\begin{equation}\label{Main theorem equation}
				\lambda_{\mathfrak{b}} (\delta) \leq \beta_{\mathfrak{b}} - C^2 \delta^2 +o(\delta^2),
			\end{equation}
			where \(C = \frac{M_3(\mathfrak{b})}{2} \), with \(M_3(\mathfrak{b})\) defined in \eqref{Third moment} and \(\beta_{\mathfrak{b}}\) in \eqref{Beta a definition}. In particular, \( \lambda_{\mathfrak{b}} (\delta)\) is an eigenvalue if \(\delta \) is small enough.
		\end{theorem}
		\begin{remark}
			The constant \(C\) originates from the interpolation of the trial state phase. Its choice is optimal in the sense that an improvement could only be achieved by choosing a different trial function construction. Moreover, it also indicates that the theorem is not valid for \(\mathfrak{b}=(b,1)\) because the \(\delta^ 2\) term becomes positive. A possible explanation could be by looking at the 1-D operator. For this operator, we have that the eigenfunction has negative derivative at the origin \cite[Theorem 1.1]{Moments}, which could be due to the fact that the eigenfunction is more localized in the part where the magnetic field is smaller. However, there is no satisfying explanation of this fact yet.
		\end{remark}
		\begin{remark}\label{Slight curve remark}
			If we consider \(\mathcal{L}_{\mathfrak{b}, \delta}\) as in \eqref{La} but now the magnetic barrier is given by the \(C^2\)-curve \(\gamma_{\delta}\) introduced in \cite{Almostflat} corresponding to a slightly curved line such that its algebraic curvature \(\kappa_{\delta}\) satisfied \(\kappa_{\delta}(s)= \delta\kappa(s)\), then an analogue result as \cite[Theorem 1.2]{Almostflat} can be obtained. This can be proven by using a trial state of the form \[\psi(s,t) = \chi (\ell^{-1} t) \phi_{\mathfrak{b}} g(s),\] where \((t,s)\) are classical tubular coordinates, \(\ell = \delta^{-\rho}\) for some \(\rho\in(0,1)\) and \(g(s)\) is chosen in a similar way as in \cite[Equation (4.8)]{Almostflat} but with respect to \(M_3(\mathfrak{b})\). 
			
			Another key step for this proof is to adapt \cite[Lemma A.1]{Semiclassicalbrokenedge} to apply a suitable change of gauge needed for the proof. The result will be of the form 
			\[\lambda_{\mathfrak{b}}(\delta) \leq \beta_{\mathfrak{b}} - \frac{M_3(\mathfrak{b})^2 }{4}\delta^2 + o(\delta^2),\]
			where again we will only get an eigenvalue, for small enough \(\delta\), if we bend the curve towards the region where the magnetic field is one. Similarly the magnetic field must have value \(b\in (-1,0)\) on the bigger region. 
		\end{remark}
		To prove Theorem \ref{Main Theorem}, it is enough to construct a trial state with energy below \(\beta_{\mathfrak{b}}\). In the next section we construct such trial state following the ideas introduced in \cite{Correggi} and applied in \cite{Almostflat}. The trial state is constructed from a 1-D operator called the ``trapping magnetic step" \cite{HislopPopoff} and using the axial symmetry of the domain. 
		
		In Section \ref{Section 2}, we collect useful results on known 1-D and 2-D operators with magnetic step fields. In Section \ref{Section 3}, we study the axial symmetry of the domain and use it to construct our trial state. In Section \ref{Estimate of the energy}, we compute the energy term for the trial state and in Section \ref{Section L^2 norm} we do the same for the \(L^2\) norm. In Section \ref{Section proof main theorem}, we combine the previous results to prove Theorem~\ref{Main Theorem}. In Section~\ref{Section applications} we give an application of this theorem by using \(\mathcal{L}_{\mathfrak{b}, \delta}\) as an effective operator for the Dirichlet realization of the magnetic Laplacian in bounded smooth domains presenting an almost flat magnetic barrier in their interior.
		
		\section{Preliminaries}\label{Section 2}
		If \(\delta = 0\), the operator \(\mathcal{L}_{\mathfrak{b}, 0}\) has been previously studied (see for example \cite[Section 2.4]{AssadKachmarSundqvist}) for particular values of \(\mathfrak{b}\).  We denote the ground state energy of \(\mathcal{L}_{\mathfrak{b},0}\) by
		\begin{equation}\label{Beta a definition}
			\beta_{\mathfrak{b}} := \inf\text{Spec} (\mathcal{L}_{\mathfrak{b},0}).
		\end{equation}
		The fact that \(\mathcal{L}_{\mathfrak{b},0}\) is invariant under translations in the \(x_1\)-direction allows us to use partial Fourier transform with respect to \(x_1\) in order to reduce the study of \(\mathcal{L}_{\mathfrak{b},0}\) to that of a family of Schr\"odinger operators, \(\mathfrak{h}_{\mathfrak{b}}[\xi]\), on \(L^2(\R)\) parameterized by \(\xi\in \R\) (see \cite{HislopPopoff}).  The operators \(\mathfrak{h}_{\mathfrak{b}}[\xi]\) are called fiber operators and are given by 
		\begin{equation}\label{1-D operator}
			\mathfrak{h}_{\mathfrak{b}}[\xi] = - \frac{d}{dt^2}+V_{\mathfrak{b}}(\xi, t),
		\end{equation}
		where 
		\[V_{\mathfrak{b}}(\xi, t) = \begin{cases}
			(b_1t + \xi)^2, & t<0\\
			(b_2 t+\xi)^2, & t>0\end{cases}
		\]
		Moreover, \[\mathcal{D}(h_{\mathfrak{b}}[\xi]) = \{u\in B^1(\R) : h_{\mathfrak{b}}[\xi] u \in L^2(\R), u'(0_+)=u'(0_-)\},\]where \(B^1(\R):= \{u\in L^2(\R) : t u\in L^2(\R) , u'\in L^2(\R)\}\). The associated quadratic form is given by
		\begin{equation}\label{quadratic form 1-D}
			q_{\mathfrak{b}}[\xi] (u) =\int_{\R} \left(|u'(t)|^2 +V_{\mathfrak{b}}(\xi, t) |u(t)|^2\right) dt,
		\end{equation}
		with \(\mathcal{D}(q_{\mathfrak{b}}[\xi]  )= B^1(\R)\). The spectra of the operators \(\mathcal{L}_{\mathfrak{b},0}\) and \(\mathfrak{h}_{\mathfrak{b}}[\xi]\) are connected as follows (see \cite[Section 4.3]{FournaisHelfferbook})
		\begin{equation}\label{Link Labd and h[xi]}
			\text{Spec}(\mathcal{L}_{\mathfrak{b},0}) = \overline{\bigcup_{\xi\in \R}\text{Spec}(\mathfrak{h}_{\mathfrak{b}}[\xi]}). 
		\end{equation}
		Because of this, we can rewrite the ground state energy of \(\mathcal{L}_{\mathfrak{b},0}\) introduced in \eqref{Beta a definition} as
		\begin{equation}\label{beta as 1-D}
			\beta_{\mathfrak{b}} =  \inf_{\xi\in \R}\mu_{\mathfrak{b}}(\xi),
		\end{equation}
		where \(\mu_{\mathfrak{b}}(\xi) = \inf \text{Spec} (\mathfrak{h}_{\mathfrak{b}}[\xi])\). 
		
		\begin{remark}\label{Remark b cases}
			For \(B>0\), if we introduce the \(L^2\)-unitary transform \cite{HislopPopoff} 
			\begin{equation*}
				U_B u(x_1, x_2) = B^{-1/2}u(B^{-1/2}x_1, B^{-1/2}x_2), 
			\end{equation*}
			for \(B>0\), we have \(U_B^{-1}\mathcal{L}_{\mathfrak{b}, \delta}U_B = B^{-1}\mathcal{L}_{B \mathfrak{b}, \delta}\). Hence, we can reduce the study of \(\mathcal{L}_{\mathfrak{b}, \delta}\), and consequently of \(\mathfrak{h}_{\mathfrak{b}}[\xi]\), to the following cases:
			\begin{enumerate}
				\item \textbf{``Magnetic wall" }\(\mathfrak{b} =(0,1) \) or \(\mathfrak{b} =(1,0)\). This case has already been studied in \cite{HislopPopoff} and we know that \( \inf {\text{Spec}}_{\text{ess}}(\mathfrak{h}_{\mathfrak{b}}[\xi]) = [\xi^2, +\infty)\), which prevents considering an eigenfunction associated to the infimum in \eqref{beta as 1-D}. 
				\item\label{Trapping magnetic step} \textbf{``Trapping magnetic step"} \(\mathfrak{b}=(1,b)\) or \(\mathfrak{b}= (b,1)\) with \(b\in(-1,0)\). By \cite[Theorem 1.1]{Moments} we know that there exists a non-degenerate minimum \(\xi_{\mathfrak{b}}\) in \eqref{beta as 1-D} and \(\mu_{\mathfrak{b}}(\xi_{\mathfrak{b}})\) is a simple eigenvalue. Moreover, the corresponding ground state \(\phi_{\mathfrak{b}}\)  satisfies \(\phi'_{\mathfrak{b}}(0)< 0 \) when  \( \mathfrak{b}= (b,1)\) and \(\phi'_{\mathfrak{b}}(0)>0 \) when \( \mathfrak{b}= (1,b)\). 
				
				Similarly, as in \cite[Chapter 2]{Raymondbook}, one can show that  \(\phi_{\mathfrak{b}}\) can be chosen to be real-valued and that it has exponential decay (see Appendix~\ref{AppendixB} for more details). In addition to this, we also know (see \cite{HislopPopoff, AssadKachmarSundqvist})
				\[\abs{b}\Theta_0 <\beta_{\mathfrak{b}} <\abs{b}, \]
				where \(\Theta_0\) is the de Gennes constant (see \cite[Section 3.2]{FournaisHelfferbook}).
				\item\label{Symmettric trapping magnetic step} \textbf{``Symmetric trapping magnetic step" }\(\mathfrak{b} =(-1,1)\) or \(\mathfrak{b}=(1,-1)\). This case can be linked to the standard de Gennes operators defined on the half-line. There exists a unique minimum \(\xi_{-1} = -\sqrt{\Theta_0}\), where \(\mu_{\mathfrak{b}}(\xi_{-1})\) is a simple eigenvalue~\cite{HislopPopoff, Moments}. In particular \( \beta_{\mathfrak{b}} = \Theta_0 \).
				\item \textbf{``Non-trapping magnetic step" }\(\mathfrak{b}=(1,b)\) or \(\mathfrak{b}= (b,1)\) with \(b\in(0,1)\). This case has already been studied in \cite{Iwatsuka}. We know that \(\beta_{\mathfrak{b}} = b\) and that \(\mu_{\mathfrak{b}}(\xi)\) does not achieve a minimum. 
			\end{enumerate}
		\end{remark}
		The techniques required to prove Theorem \ref{Main Theorem} rely on the existence of a minimum in \eqref{beta as 1-D} corresponding to a simple eigenvalue, since we are using the associated eigenfunction in the construction of the trial state. Due to this, it is clear that only cases \ref{Trapping magnetic step} and \ref{Symmettric trapping magnetic step} can be considered.}
	
	We will need the first and third moments of \(\mathcal{L}_{\mathfrak{b}, 0}\) in the proof of Theorem \ref{Main Theorem}. We recall the definition of the \(n\)-th moment,
	\begin{equation}\label{Moment definition}
		M_n(\mathfrak{b}) = \int_{-\infty}^{+\infty} \frac{1}{\sigma_{\mathfrak{b}}(t)} \paren{\xi_{\mathfrak{b}} + \sigma_{\mathfrak{b}} (t)t}^n \abs{\phi_{\mathfrak{b}}(t)}^2 \, dt.
	\end{equation}
	where \(\sigma_{\mathfrak{b}}(t):= b_1 \mathbbm{1}_{\R_-}(t) + b_2 \mathbbm{1}_{\R_+}(t)\)  and \(\xi_{\mathfrak{b}}\) is the unique minimum of \(\xi\mapsto \mu_{\mathfrak{b}}(\xi)\).
	\begin{prop} \cite[Proposition 4.1]{Moments}
		If \(\mathfrak{b} =(b,1)\) with \(b\in [-1,0)\), then 
		\begin{equation}\label{First moment}
			M_1(\mathfrak{b}) =0 ,
		\end{equation}
		and
		\begin{equation}\label{Third moment}
			M_3(\mathfrak{b}) = \frac{1}{3}\bigg (\frac{1}{b}-1\bigg ) \xi_{\mathfrak{b}} \phi_{\mathfrak{b}}(0)\phi'_{\mathfrak{b}}(0),
		\end{equation}
		where \(\phi_{\mathfrak{b}}\) denotes the \(L^2\)-normalized ground state of \(\mathfrak{h}_{\mathfrak{b}}[\xi_{\mathfrak{b}}]\). 
	\end{prop}
	\section{Trial state definition}\label{Section 3}
	The domains \(\Omega_{1,\delta}\) and \(\Omega_{2, \delta}\) are similar to the domain \(\Omega_{\delta}\) considered in \cite{Almostflat}. In each of these regions, the magnetic field is constant, so we can mimic the trial state construction done for \(\Omega_{\delta}\) but using \(\phi_{\mathfrak{b}}\) instead of the eigenfunction associated to the ground state of the de Gennes operator.
	\subsection{Domain axial symmetry and trial state}
	The domain presents  axial symmetry with respect to the line \(x_1 = -x_2 \tan(\delta/2) \). The matrix corresponding to the reflection with respect to this axis is given by
	\[S_{\delta}  =  \begin{pmatrix}
		-\cos{\delta}  & -\sin{\delta}  \\
		-\sin{\delta}& \cos{\delta} 
	\end{pmatrix}.\]
	\begin{prop}\label{Proposition axial symmetry}\textbf{(axial symmetry)} If \(\mathcal{L}_{\mathfrak{b},  \delta}\) admits an smallest eigenvalue, then the associated eigenfunction \(u\) can be chosen so that it will respect axial symmetry
		\begin{equation}\label{axial symmetry}
			u(x) = \overline{u(S_{\delta} (x) )}, \; \forall x\in \R^2. 
		\end{equation}
		\begin{proof}
			This follows by a direct calculation, using \(S_{\delta}(\Omega_{1, \delta})= \Omega_{1, \delta}\) and \(S_{\delta} (\Omega_{2, \delta}) = \Omega_{2, \delta}\). See \cite[Proposition 3.6]{Bonnthesis} for the same calculation in the sector.
		\end{proof}
	\end{prop}
	We will use this symmetry property in the construction of our trial state.  Before that, we introduce some notation. Let 
	\begin{equation}\label{Definition gamma, delta, l}
		\gamma := \delta^{\frac{1}{2}}\qquad \text{ and }\qquad  \ell := \delta^{-\frac{1}{2}}.
	\end{equation}
	As stated before, we will consider \(\delta\) small and positive. That means that \(\gamma\) will be small (but larger than \(\delta\), in particular \(\gamma >2\delta\) for \(\delta < \frac{1}{4}\)) and \(\ell\) will be very large quantity going to \(+\infty\) when \(\delta\rightarrow 0\). 
	
	The idea is to divide the domain into two parts \(T^+\) and \(T^-\) that obey the symmetry. By this we mean that \(S_{\delta}(T^+)= T^-\), and so we can use Proposition \ref{Proposition axial symmetry}. To simplify the calculations we will introduce a cutoff in the \(x_2\) direction which will be parameterized by \(\ell\). We will consider the regions \(V^+\) and \(V^-\) to glue the phase of the trial state between \(T^ +\) and \(T^-\) to ensure we have a continuous function. This division is depicted in Figure \ref{fig1}.
	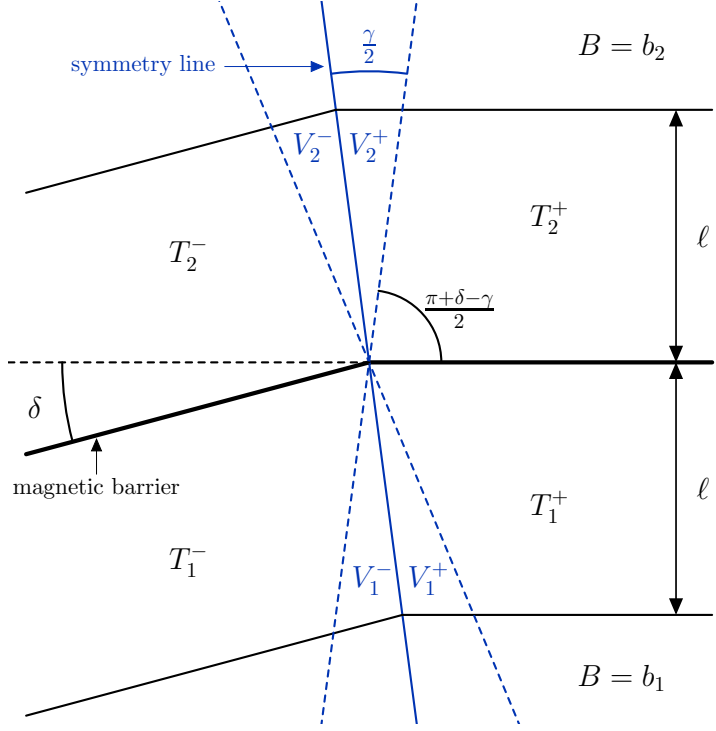
\begin{figure}[ht!]
		\scalebox{0.8}{
			\begin{tikzpicture}[line cap=round,line join=round,>=triangle 45,x=0.6cm,y=0.6cm]
				\clip(-10,-10) rectangle (10,10);
				\draw [line width=1.pt,color=bblue, domain=-10:10] plot(\x,{\x /-0.13165}); 
				\draw [line width=1.pt, domain=-0.92155:9.5] plot(\x,7); 
				\draw [line width=1.pt, domain=0.92155:9.5] plot(\x,-7); 
				\draw [line width=2.pt, domain=0:9.5] plot(\x,0); 
				\draw [line width=2.pt, domain=-9.5:0] plot(\x,{\x* 0.26794919243});
				\draw [line width=1.pt, domain=-9.5:-0.92155] plot(\x,{\x* 0.26794919243 +7.24693});
				\draw [line width=1.pt, domain=-9.5:0.92155] plot(\x,{\x* 0.26794919243 - 7.24693});
				\draw [dashed, line width=1.pt,color=bblue, domain=-10:10] plot(\x,{\x /-0.41421356237}); 
				\draw [dashed, line width=1.pt,color=bblue, domain=-10:10] plot(\x,{\x /0.13165}); 
				\draw [dashed, line width=1.pt, domain=-10:0] plot(\x,0); 
				\draw [line width=1.pt] (-8.5,0) arc (180: 195: 8.5); 
				\draw[color= black] (-9.25, -1.25) node {\Large \(\delta\)}; 
				\draw[->] (-7.55, -3.1) -- (-7.55, -2.1); 
				\draw[color = bblue, ->] (-4, 8.2) -- (-1.25, 8.2); 
				\draw[color= black] (-7.55, -3.5) node {\text{magnetic barrier}}; 
				\draw [<->,line width=0.8 pt] (8.5,0.) -- (8.5,7.);
				\draw[color= black] (9.25, 3.5) node {\Large \(\ell\)}; 
				\draw [<->,line width=0.8 pt] (8.5,0.) -- (8.5,-7.);
				\draw[color= black] (9.25, -3.5) node {\Large \(\ell\)}; 
				\draw [line width=1.pt] (2,0) arc (0: 1980/24: 2); 
				\draw[color= black] (2.5, 1.5) node {\Large \(\frac{\pi+\delta-\gamma}{2}\)}; 
				\draw [color=bblue, line width=1.pt] (1.0532,8) arc (1980/24: 1980/24 +15: 8.0690291); 
				\draw[color= bblue] (0, 8.75) node {\Large \(\frac{\gamma}{2}\)}; 
				\draw[color= black] (7, 8.75) node {\Large \(B=b_2\)}; 
				\draw[color= black] (7, -8.75) node {\Large \(B=b_1\)}; 
				\draw[color= bblue] (-6.25, 8.25) node { \text{symmetry line}}; 
				\draw[color= black] (5, 4) node {\Large \(T^+_{2}\)}; 
				\draw[color= black] (5, -4) node {\Large \(T^+_{1}\)}; 
				\draw[color= bblue] (0, 6) node { \Large \(V^+_{2}\)}; 
				\draw[color= bblue] (-1.55, 6) node {\Large \(V^-_{2}\)}; 
				\draw[color= bblue] (0.1, -6) node { \Large \(V^-_{1}\)}; 
				\draw[color= bblue] (1.65, -6) node { \Large \(V^+_{1}\)}; 
				\draw[color= black] (-5, 3) node {\Large \(T^-_{2}\)}; 
				\draw[color= black] (-5, -5.5) node {\Large \(T^-_{1}\)}; 
			\end{tikzpicture}
		}
		\caption{In this figure, we show how the domain is partitioned according to Definition \ref{Set partition}.}\label{fig1}
	\end{figure}
	\begin{defin}\label{Set partition}
		Let \(T^{\pm} :=T^{\pm}_{1} \cup T^{\pm}_{2}  \) and  \(V^{\pm} :=V^{\pm}_{1} \cup V^{\pm}_{2}  \) with 
		\[T^+_{1} = \bigg \{(x_1, x_2) \in (0, +\infty) \times (-\ell ,0) : -x_2\tan(\frac{\gamma+\delta}{2}) <x_1 \bigg \},\]
		\[T^+_{2} = \bigg \{(x_1, x_2) \in (0, +\infty) \times (0,\ell) : x_2 \tan\left(\frac{\gamma - \delta}{2}\right) <x_1\bigg \},\]
		\[ T^-_{1} = S_{\delta} (T^+_{1})\qquad \text{and} \qquad T^-_{2} = S_{\delta} (T^+_{2}).\]
		Let \(V^+_{2} \) be the sector defined in polar coordinates by  \(0<r < r_{\ast}:= \frac{\ell }{\abs{\sin{\theta}}}\) and \(\theta\in \big (\frac{\pi+\delta- \gamma}{2}, \frac{\pi+\delta}{2}\big )\). Similarly, let \(V^+_{1} \) the one with \(\theta\in \big (\frac{3\pi+\delta}{2}, \frac{3\pi+\delta+\gamma}{2}\big )\). Moreover, consider
		\[V^-_{1} := S_{\delta} (V^+_{1})\text{, and }  V^-_{2} := S_{\delta} (V^+_{2}).\]
	\end{defin}
	In order to simplify the computations it is useful to apply a change of Gauge to \(\mathcal{L}_{\mathfrak{b}, \delta}\). Consider the function \( \upzeta \in H^1_{\mathrm{loc}}(\R^2)\) defined as
	\begin{equation}\label{First change of Gauge}
		\upzeta(x) := \begin{cases}
			\frac{1}{2}\mathbbm{1}_{\R_-}(x_1) (b_1-b_2) x_1^2 \tan{\delta} & \text{if }(x_1, x_2)\in \Omega_{1,\delta}\\
			0 & \text{if }(x_1, x_2)\in \Omega_{2,\delta}.
		\end{cases} 
	\end{equation}
	Note that \(\mathcal{A}_{\mathfrak{b}, \delta} =\sigma_{\mathfrak{b}, \delta}\mathcal{A} + \nabla \upzeta \) where \(\mathcal{A}(x_1, x_2) = (- x_2 , 0 ) \), and 
	\begin{equation}
		\sigma_{\mathfrak{b}, \delta}(x_1, x_2) = \begin{cases}
			b_1 & \text{if } (x_1, x_2) \in \Omega_{1, \delta}\\
			b_2 & \text{if } (x_1, x_2) \in \Omega_{2, \delta} .
		\end{cases}
	\end{equation}
	Thus, using \cite[Lemma 1.1]{LeinfelderArticle}, we can consider the quadratic form 
	\begin{equation*}
		\tilde{\mathcal{Q}}_{\mathfrak{b}, \delta}(v) = \int_{\R^2} |(\nabla - i\sigma_{\mathfrak{b}, \delta}\mathcal{A} ) v|^2 \, dx,
	\end{equation*}
	with domain \(\mathcal{D}(\tilde{\mathcal{Q}}_{\mathfrak{b}, \delta}) = \{v\in L^2(\R^2) : (\nabla - i\sigma_{\mathfrak{b}, \delta}\mathcal{A} )v \in L^2(\R^2 , \R^2)\}\). The Gauge invariance implies that 
	\begin{equation}\label{Quadratic gauge invariance}
		\tilde{\mathcal{Q}}_{\mathfrak{b}, \delta}(v) =  \mathcal{Q}_{\mathfrak{b}, \delta}(e^{i\upzeta}v),
	\end{equation}
	for all \(v\in \mathcal{D}(\tilde{\mathcal{Q}}_{\mathfrak{b}, \delta})  \). 
	\begin{lemma}\label{Lemma 1}
		Let \(u_+\) be a function such that 
		\begin{equation*}
			\int_{T^+} |(\nabla - i\sigma_{\mathfrak{b}, \delta} \mathcal{A})u_+|^2 dx < +\infty.
		\end{equation*}
		Consider
		\begin{equation}
			\varphi(y) :=-\frac{\sin{\delta}}{2}y^T S_{\delta} y =  \frac{\sin(2\delta)}{4} (y_1^2 -y_2^2) +y_1y_2 \sin^2{\delta},
		\end{equation}
		and 
		\begin{equation}\label{u_- definition}
			u_- (x) = \begin{cases}
				e^{-ib_1 \varphi(S_{\delta}x)}\overline{u_+(S_{\delta}(x))} & \text{ if } x\in T_{1}^- \\
				e^{-ib_2\varphi(S_{\delta}x)}\overline{u_+(S_{\delta}(x))} & \text{ if } x\in T_{2}^- 
			\end{cases} ,
		\end{equation}
		then we have
		\begin{equation}
			\int_{T^+} |(\nabla - i\sigma_{\mathfrak{b}, \delta} \mathcal{A})u_+|^2 dx = \int_{T^-} |(\nabla - i\sigma_{\mathfrak{b},\delta} \mathcal{A})u_-|^2 dx.
		\end{equation}
	\end{lemma}
	\begin{remark}
		Note that by construction  \(\varphi(S_{\delta}(y)) = \varphi(y)\) since \(S_{\delta}^T = S_{\delta}\) and \(S_{\delta}^2 = \text{Id}\). 
	\end{remark}
	\begin{proof}[Proof of Lemma \ref{Lemma 1}]
		We follow the proof in \cite[Lemma 3.1]{Almostflat}. Since \(\mathcal{A}=(-x_2, 0)\), then using the change of variables \(y= S_{\delta} x\),
		\begin{multline*}
			\int_{T^-} |(\nabla -i\sigma_{\mathfrak{b},\delta} \mathcal{A})u_-|^2 \,  dx \\ 
			=  \int_{T^+} |S_{\delta}(S_{\delta}\nabla_y -i\underbrace{\sigma_{\mathfrak{b},\delta} (S_{\delta}^{-1}(y))}_{=\sigma_{\mathfrak{b},\delta}(y)} \mathcal{A}(S_{\delta}^{-1}(y)))u_-(S_{\delta}^{-1}(y))|^2 \, dy \\
			=\int_{T^+} |(\nabla_y-i\sigma_{\mathfrak{b}, \delta}(y)\Tilde{\mathcal{A}}(y))u_-(S_{\delta}(y))|^2 dy,         
		\end{multline*}
		where we have used \(S_{\delta} = S_{\delta}^{-1} = S_{\delta}^{\ast}\) and
		\[\Tilde{\mathcal{A}}(y) = S_{\delta} \mathcal{A}S_{\delta}(y) = \left(\begin{array}{c}
			-\frac{y_1}{2}\sin (2\delta) +y_2 \cos^2{\delta}\\
			-y_1 \sin^2{\delta}   + \frac{y_2}{2}\sin(2\delta)
		\end{array}\right).  \]
		Note that \(\text{curl}(\Tilde{\mathcal{A}})= -1\) and \(\Tilde{\mathcal{A}}+ \nabla \varphi =(y_2, 0)\). Thus, by Gauge invariance, 
		\begin{equation*}
			\begin{split}
				\int_{T_{2}^-} |(\nabla -i b_2\mathcal{A})u_-|^2 dx & =  \int_{T_{2}^+} \bigg |\left(\nabla_y -i \left(\begin{array}{c}
					b_2 y_2  \\
					0 
				\end{array}\right)\right) e^{ib_2\varphi(y)} u_-(S_{\delta} (y))\bigg|^2 \, dy \\
				&= \int_{T_{2}^+} \bigg |\left(\nabla_y -i \left(\begin{array}{c}
					-b_2y_2  \\
					0 
				\end{array}\right)\right) e^{-ib_2\varphi(y)} \overline{u_-(S_{\delta} (y)})\bigg|^2 \, dy  \\
				& =  \int_{T_{2}^+} |(\nabla_y -ib_2 \mathcal{A}) u_+(y)|^2 \, dy ,
			\end{split}
		\end{equation*}
		since \(\overline{u_- (x)} =e^{-ib_2\varphi(S_{\delta}x)}u_+(S_{\delta}(x)) \). Similarly, 
		\[ \int_{T_{1}^-} |(\nabla -i b_1 \mathcal{A})u_-|^2 \,  dx  = \int_{T_{1}^+} |(\nabla_y -ib_1 \mathcal{A}) u_+(y)|^2 \, dy.\qedhere\]
	\end{proof}
	We will next construct a trial state that fits into Lemma \ref{Lemma 1}. To simplify computations it will be helpful if our trial state is supported in the sets introduced in Definition \ref{Set partition}. We therefore introduce the following cut-off:
	\begin{defin}\label{Definition of phi_l}
		Let \(\chi\in C_c^{\infty}(\R)\) be a cut-off function such that \(0\leq \chi\leq 1\), \(\text{supp } \chi\subset [-1, 1]\), \(\chi =1\) on \([-\frac{1}{2}, \frac{1}{2}]\), and consider
		\begin{equation}\label{Phi_b,l definition}
			\phi_{\mathfrak{b}, \ell }(t) := \chi\paren*{\frac{t}{\ell}}\phi_{\mathfrak{b}}(t),
		\end{equation}
		where \(\phi_{\mathfrak{b}}\) denotes the eigenfunction introduced in Remark \ref{Remark b cases}. 
	\end{defin}
	Note that, for large \(\ell\), due to the exponential decay of \(\phi_{\mathfrak{b}}\) and its derivative, 
	\begin{align}\label{Decay of phi_b,l}
		\begin{split}
			\int_{-\infty}^{+\infty} \abs{\phi_{\mathfrak{b}, \ell}}^2\, dt &= 1 +\mathcal{O}(\ell^{-\infty})\\
			q_{\mathfrak{b}}[\xi_{\mathfrak{b}}]( \phi_{\mathfrak{b}, \ell}) &=  q_{\mathfrak{b}}[\xi_{\mathfrak{b}}]( \phi_{\mathfrak{b}})+\mathcal{O}(\ell^{-\infty})\\
			\int_{-\infty}^{+\infty} \frac{1}{\sigma_{\mathfrak{b}}(t)} \paren{\xi_{\mathfrak{b},\ell} + \sigma_{\mathfrak{b}} (t)t}^n \abs{\phi_{\mathfrak{b},\ell}(t)}^2 \, dt & = M_n(\mathfrak{b}) +\mathcal{O}(\ell^{-\infty}), 
		\end{split}
	\end{align}
	where \(\mathcal{O}(\ell^{-\infty})\) stands for \(\mathcal{O}(l^{-N})\) for all \(N\in \N\) as \(\ell\rightarrow +\infty\). Since \(\ell = \delta^{-\frac{1}{2}}\) by \eqref{Definition gamma, delta, l}, we will write \(\mathcal{O}(\delta^{\infty})\) instead of \(\mathcal{O}(\ell^{-\infty})\).  The main idea is to construct a trial state of the form \(\phi_{\mathfrak{b},l} (x_2) e^{i \xi_{\mathfrak{b},l} x_1}\) in \(T^+ \) and use \eqref{u_- definition} in \(T^-\). To make sure that our trial state is in the domain of the quadratic form \(\mathcal{Q}_{\mathfrak{b},\delta}\) we will introduce a cut-off function in the \(x_1\)-direction and consider an appropriate phase in the transition areas \(V^{\pm}\).
	
	For \(x\in \R^2\), let 
	\begin{equation}\label{Trialstate}
		\Psi^{tr}(x) := \eta(x) \Psi(x),
	\end{equation}  
	where \(0\leq \abs{\eta(x)}\leq 1\), 
	\begin{equation}\label{Eta definition}
		\eta(x) = \begin{cases}
			\eta_+(x_1) & \text{ if } x\in T^+\\
			\eta_+(-x_1\cos{\delta} - x_2\sin{\delta}) & \text{ if } x\in T^-
		\end{cases},
	\end{equation}
	and \(\eta_+\in H^1(\R)\) is equal to \(1\) in a small neighborhood of \(0\). Note that \(\eta(S_{\delta}x) = \eta(x)\) for \(x\in T^{\pm}\). The other term is given by
	\[\Psi(x)= \begin{cases}
		\phi_{\mathfrak{b},\ell} (x_2) e^{i \xi_{\mathfrak{b}} x_1} & \text{if }(x_1, x_2) \in T^+\\
		\phi_{\mathfrak{b},\ell} (x_2) e^{i \alpha } & \text{if }(x_1, x_2) \in V^+\\
		\phi_{\mathfrak{b},\ell} (-\sin{\delta}x_1 + \cos{\delta}x_2) e^{i \alpha} & \text{if }(x_1, x_2) \in V^-\\
		\phi_{\mathfrak{b},\ell} (-\sin{\delta}x_1 + \cos{\delta}x_2) e^{-i\xi_{\mathfrak{b}} (x_1\sin{\delta} +x_2\cos{\delta}) -ib_1  \varphi(x)} & \text{if }(x_1, x_2) \in T_{1}^-\\
		\phi_{\mathfrak{b},\ell} (-\sin{\delta}x_1 + \cos{\delta}x_2) e^{-i\xi_{\mathfrak{b}}(x_1\sin{\delta} +x_2\cos{\delta}) -ib_2 \varphi(x)} & \text{if }(x_1, x_2)\in T_{2}^-\\
		0 & \text{elsewhere} ,
	\end{cases}\]
	where \(\alpha\) is defined as follows
	
	\begin{enumerate}
		\item If \( (x_1, x_2)\in V^-\), then \(\alpha\) in polar coordinates is given by
		\begin{equation*}
			\alpha(r, \theta) := d_1  r^2 - h_1(\theta) (c_1 r -d_1 r^2),
		\end{equation*}
		with \(c_1 = \xi_{\mathfrak{b}} \sin(\frac{\gamma+\delta }{2}) \), \(d_1=  \frac{a}{4}\sin{\delta}\cos(\gamma)\), \(h_1(\theta) = -\frac{2}{\gamma}\theta + \frac{3\pi+\delta+\gamma}{2}\) for \(\theta \in \big (\frac{3\pi+\delta-\gamma}{2},  \frac{3\pi+\delta+\gamma}{2}\big )\).  
		\item If \( (x_1, x_2)\in V^+\), then \(\alpha\) in polar coordinates is given by
		\begin{equation*}
			\alpha(r, \theta) := d_2  r^2 - h_2 (\theta) (c_2 r -d_2 r^2),
		\end{equation*}
		with \(c_2= \xi_{\mathfrak{b}} \sin(\frac{\gamma-\delta }{2}) \), \(d_2=  \frac{b}{4}\sin{\delta}\cos(\gamma)\) and \(h_2(\theta) = \frac{2}{\gamma}\theta  + \frac{\pi+\delta- \gamma}{2}\) for \(\theta\in \big (\frac{\pi+\delta-\gamma}{2}, \frac{\pi+\delta+\gamma}{2}\big)\). 
	\end{enumerate}

	\begin{remark}
		The choice of \(\alpha\) is done to guarantee the continuity of \(\Psi\) (in the same spirit as in \cite[Section 3.3]{Almostflat}). 
		
		The precise expression of \(\eta\) will be given later.
	\end{remark}
	
	\section{Estimate of the energy}\label{Estimate of the energy}
	To prove Theorem \ref{Main Theorem} it suffices to show that \(\mathcal{Q}_{\mathfrak{b}, \delta}(\Psi^{tr}) -  \beta_{\mathfrak{b}} \norm{\Psi^{tr}}^2_{L^2(\R^2)}\) is negative if \(\delta\) is small enough. We start by estimating \(\mathcal{Q}_{\mathfrak{b}, \delta}(\Psi^{tr}) \) in the different regions introduced in Definition \ref{Set partition}.  
	\subsection{Energy in \(V^+\)}
	To make things more digestible, we will introduce every result as a series of lemmas.
	\begin{lemma}\label{Lemma 3} It holds that
		\begin{equation}\label{Eq 4.2}
			\begin{split}
				\int_{V^+_1} |(\nabla -ib_1\mathcal{A})\Psi |^2 dx  =&  \int_{\frac{3\pi+\delta }{2}}^{\frac{3\pi+ \delta+\gamma}{2}} \int_0^{\frac{-\ell}{\sin{\theta}}}| \phi'_{\mathfrak{b},\ell}(r\sin{\theta})|^2 r\, drd\theta \\
				& + \int_{\frac{3\pi+\delta }{2}}^{\frac{3\pi+ \delta+\gamma}{2}} \int_0^{\frac{-\ell}{\sin{\theta}}}F_1(r, \theta, h_1) |\phi_{\mathfrak{b},\ell}(r\sin{\theta})|^2r\, drd\theta , 
			\end{split}
		\end{equation} 
		and
		\begin{equation}\label{Eq 4.1}
			\begin{split}
				\int_{V^+_{2}} |(\nabla -ib_2\mathcal{A})\Psi |^2 dx   =&\int_{\frac{\pi+\delta -\gamma}{2}}^{\frac{\pi+ \delta}{2}}  \int_0^{\frac{\ell}{\sin{\theta}}}  |\phi'_{\mathfrak{b},\ell}(r\sin{\theta})|^2 \, r\, dr d\theta \\
				& +\int_{\frac{\pi+\delta -\gamma}{2}}^{\frac{\pi+ \delta}{2}} \int_0^{\frac{\ell}{\sin{\theta}}} F_2(r, \theta, h_2) |\phi_{\mathfrak{b},\ell}(r\sin{\theta})|^2 \, r\, dr d\theta ,
			\end{split}
		\end{equation} 
		where, for \(j\in \{1,2\}\), 
		\begin{equation}\label{Fj definition}
			\begin{split}
				F_j(r, \theta, h_j) &= \bigg( \frac{1}{2} b_j r(-1+\cos(2\theta)) -h'_j(\theta) (c_j-d_j r) \bigg )^2 \\
				&\quad + \bigg (\frac{r}{2}b_j\sin(2\theta) +2d_j r - h_j(\theta) (c_j -2d_jr)\bigg )^2.
			\end{split}
		\end{equation}
	\end{lemma}
	\begin{proof}
		The proof is by a direct computation in polar coordinates. Note that
		\[\nabla = \partial_r \hat{e}_r +\frac{1}{r}\partial_{\theta} \hat{e}_{\theta} \qquad\text{ and } \qquad \mathcal{A}_0 = \frac{r}{2}\hat{e}_{\theta},\]
		where \(\mathcal{A}_0 = \frac{1}{2}(-x_2, x_1)\). Since \(\sigma_{\mathfrak{b}, \delta} \mathcal{A}+  \sigma_{\mathfrak{b}, \delta} \nabla \zeta = \sigma_{\mathfrak{b}, \delta} \mathcal{A}_0\) with \(\zeta(x_1, x_2) =x_1x_2/2\), by Gauge invariance we can consider
		\[\Tilde{\Psi}(x_1, x_2) = e^{i\sigma_{\mathfrak{b}, \delta} (x_1,x_2) \frac{x_1x_2}{2}}\Psi(x_1, x_2) =\phi_{\mathfrak{b},\ell}(r\sin{\theta}) e^{i(\alpha(r, \theta) + \frac{1}{4}\sigma(x_1, x_2) r^2 \sin(2\theta))}.\]
		Having this in mind, we can compute
		\begin{equation}\label{First quantity}
			|\partial_r \Tilde{\Psi}|^2 =  \sin^2{\theta} | \phi'_{\mathfrak{b}, \ell}(r\sin{\theta})|^2+\left(\sigma_{\mathfrak{b}, \delta}\frac{r}{2}\sin(2\theta) +\partial_r\alpha(r, \theta)\right)^2 |\phi_{\mathfrak{b},\ell}(r\sin{\theta})|^2, 
		\end{equation}
		\begin{equation}\label{Second quantity}
			\begin{split}
				\bigg |\bigg(\partial_{\theta}-i\sigma_{\mathfrak{b}, \delta} \frac{r^2}{2}\bigg) \Tilde{\Psi}\bigg |^2 & =  r^2\cos^2{\theta} | \phi'_{\mathfrak{b},\ell }(r\sin{\theta})|^2 \\
				&+(\partial_{\theta}\alpha(r,\theta) +\frac{1}{2}r^2 \sigma_{\mathfrak{b}, \delta} (-1 + \cos(2\theta) ))^2 |\phi_{\mathfrak{b},\ell}(r\sin{\theta})|^2.
			\end{split}
		\end{equation}
		Note that we have used the fact that \(\phi_{\mathfrak{b},\ell}\) and \(\alpha\) are real-valued. For \((x_1, x_2) \in V^+_{2}\), we have in polar coordinates
		\[\partial_r\alpha(r,\theta) =2d_2 r - h_2(\theta) (c_2 -2d_2 r) , \]
		and
		\[\partial_{\theta}\alpha(r,\theta) = -h'_2(\theta)(c_2r-d_2r^2). \]
		Taking this into account, by \eqref{First quantity} and \eqref{Second quantity} we have
		\[\begin{split}
			\int_{V^+_{2}} |(\nabla -ib_2\mathcal{A})\Psi |^2 dx  = & \int_{\frac{\pi+\delta -\gamma}{2}}^{\frac{\pi+ \delta}{2}} \int_0^{\frac{\ell}{\sin{\theta}}}  \left(|\partial_r \Tilde{\Psi}|^2 +r^{-2}\bigabs{(\partial_{\theta}-ib_2\frac{r^2}{2}) \Tilde{\Psi}}^2 \right) r\, drd\theta \\
			=&\int_{\frac{\pi+\delta -\gamma}{2}}^{\frac{\pi+ \delta}{2}}  |\phi'_{\mathfrak{b},\ell}(r\sin{\theta})|^2 r\, dr d\theta \\
			&+\int_{\frac{\pi+\delta -\gamma}{2}}^{\frac{\pi+ \delta}{2}} \int_0^{\frac{\ell}{\sin{\theta}}} F_2(r, \theta, h_2) |\phi_{\mathfrak{b},\ell}(r\sin{\theta})|^2 r\, dr d\theta .
		\end{split}\]
		For \(V_1^+\), the proof is analogous.
	\end{proof}
	\begin{lemma}\label{Lemma 4} We have, as \(\delta \to 0^+\),
		\begin{equation}\label{Eq 4.3}
			\int_{V^+_{2}} |(\nabla  -ib_2 \mathcal{A})\Psi |^2 dx  = \frac{1}{2}\delta^{\frac{1}{2}} J_{b_2} + \mathcal{O}(\delta^{\frac{3}{2}}),
		\end{equation}
		with 
		\begin{equation}
			\begin{split}\label{J1}
				J_{b_2} :=& \int_0^{+\infty}\left(| \phi'_{\mathfrak{b}}(t)|^2  +(b_2t+\xi_{\mathfrak{b}})^2 |\phi_{\mathfrak{b},\ell} (t)|^2 \right) t\, dt\\
				&  -\delta^{\frac{1}{2}} \int_0^{+\infty} (b_2t+\xi_{\mathfrak{b}}) (b_2t+2\xi_{\mathfrak{b}}) |\phi_{\mathfrak{b}} (t)|^2\, t\, dt . 
			\end{split}
		\end{equation}
	\end{lemma}
	\begin{proof}
		We are considering \(\delta\) to be a small positive quantity, and we know that \(\gamma = \delta^{\frac{1}{2}}\) by \eqref{Definition gamma, delta, l} . If \(\theta\in \paren*{\frac{\pi+\delta -\gamma}{2}, \frac{\pi+\delta}{2} }\), we have
		\begin{equation*}
			\sin{\theta} = 1+\mathcal{O}(\delta) , \quad  \sin{2\theta} = \mathcal{O}(\delta^{\frac{1}{2}} ), \quad \frac{1}{\sin^2{\theta}} =  1+\mathcal{O}(\delta ),
		\end{equation*}
		\begin{equation}\label{c_2 and d_2 expansion}
			c_2 =  \xi_{\mathfrak{b}} \left(\frac{\gamma-\delta }{2}\right)+ \mathcal{O}(\delta^{\frac{3}{2}}),\, d_2 = \frac{b_2  \delta}{4}+\mathcal{O}(\delta^{\frac{3}{2}}),
		\end{equation}
		and by construction \(h'_2(\theta) = \frac{2}{\gamma}\). Let
		\[\int_{V^+_{2}} |(\nabla -ib_2 \mathcal{A})\Psi |^2 dx  = I_1 + I_2,\]
		where 
		\[I_1 = \int_{\frac{\pi+\delta -\gamma}{2}}^{\frac{\pi+ \delta}{2}} \int_0^{\frac{\ell}{\sin{\theta}}}| \phi'_{\mathfrak{b},\ell}(r\sin{\theta})|^2  r\, drd\theta\]
		and 
		\[I_2 = \int_{\frac{\pi+\delta -\gamma}{2}}^{\frac{\pi+ \delta}{2}} \int_0^{\frac{\ell}{\sin{\theta}}}F_{2 }(r, \theta, h_2) |\phi_{\mathfrak{b},\ell}(r\sin{\theta})|^2 r\, drd\theta.\]
		\(\bullet \) \textbf{\(I_1\):} Introducing the change of variable \(t = r\sin{\theta}\), \(dt =  \sin{\theta}\,dr\), we obtain 
		\begin{equation}\label{I_1 definition}
			\begin{split}
				I_1 &= \int_{\frac{\pi+\delta -\gamma}{2}}^{\frac{\pi+ \delta}{2}} \frac{1}{\sin^2{\theta}} \int_0^{\ell}| \phi'_{\mathfrak{b},\ell}(t)|^2  t\, dt\, d\theta  \\
				&= \frac{\gamma}{2}\int_0^{+\infty}| \phi'_{\mathfrak{b}}(t)|^2  t\, dt + \mathcal{O}(\delta^{\infty}) + \mathcal{O}(\delta^{\frac{3}{2}}),
			\end{split}
		\end{equation}
		where we have used \eqref{Definition gamma, delta, l} and \eqref{Decay of phi_b,l}. The \(\mathcal{O}(\delta^{\frac{3}{2}})\) term is the first correction term in the \(\theta\) integral, and the \(\mathcal{O}(\delta^{\infty})\) follows from \eqref{Decay of phi_b,l} and \(\ell = \delta^{-\frac{1}{2}}\). 
		
		\(\bullet\) \textbf{\(I_2\):} Before starting with explicit computations of the integral, we look closely at the second summand in \(F_2(r, \theta, h_2)\). Recall that \(-1\leq h_2(\theta)\leq 0\) in \(V^+_{2}\) and, by \eqref{c_2 and d_2 expansion}, the quantities \(c_2 \) or \(d_2\) are of order \(\mathcal{O}(\delta^{\frac{1}{2}})\) and \(\mathcal{O}(\delta)\) respectively. This implies that 
		\begin{equation}\label{Eq 4.7}
			\bigg (b_2\frac{r}{2}\sin(2\theta) +2d_2 r - h_2(\theta) (c_2 -2d_2 r)\bigg )^2  =  \mathcal{O}(\delta )(r^2 +r +1).
		\end{equation} 
		This term together with the integral in \(\theta\) and properties of \(\phi_{\mathfrak{b}, \ell}\) stated in \eqref{Decay of phi_b,l} will give a term of order \(\mathcal{O}(\delta^{\frac{3}{2}})\). For the first term in \(F_2(r, \theta, h_2) \), after the same change of variable \(t=r\sin{\theta}\), we obtain
		\[\int_{\frac{\pi+\delta -\gamma}{2}}^{\frac{\pi+ \delta}{2}} \int_0^{\ell}\left(tb_2\sin{\theta} +h'_2(\theta)\bigg (c_2 -d_2\frac{t}{\sin{\theta}}\bigg ) \right)^2 |\phi_{\mathfrak{b},\ell}(t)|^2 \frac{1}{\sin^2{\theta}} t\, dt\,d\theta .\]
		Since \(\frac{1}{\sin^{n}{\theta}} = 1+\mathcal{O}(\delta)\) for \(n\in \{1, \ldots, 4\}\), \(h'_2(\theta) = \frac{2}{\gamma}\) and \eqref{c_2 and d_2 expansion} we can rewrite 
		\begin{equation*}
			I_2 =  \int_{\frac{\pi+\delta -\gamma}{2}}^{\frac{\pi+ \delta}{2}} \int_0^{\ell}  \paren*{b_2 t +   \xi_{\mathfrak{b}} \bigg(1-\frac{\delta}{\gamma}\bigg ) - \frac{b_2\delta}{2\gamma}t}^2 |\phi_{\mathfrak{b},\ell}(t)|^2 \, t \, dt d\theta  + \mathcal{O}(\delta^{\frac{3}{2}}), 
		\end{equation*}
		Observe that, using \eqref{Decay of phi_b,l}, we can rewrite 
		\begin{equation}\label{I_2 definition}
			I_2 =\frac{\gamma}{2}  \int_0^{\infty} \left(b_2t +\left(\xi_{\mathfrak{b}}\left(1-\frac{\delta}{\gamma}\right) - \frac{b_2\delta}{2\gamma} t\right)\right)^2 |\phi_{\mathfrak{b}}(t)|^2 t dt + \mathcal{O}(\delta^{\infty})+\mathcal{O}(\delta^{\frac{3}{2}}).
		\end{equation}
		Summing up \eqref{I_1 definition} and \eqref{I_2 definition} we obtain 
		\begin{equation*}
			\begin{split}
				\int_{V^+_{2}} |(\nabla -ib_2\mathcal{A})\Psi |^2 \,  dx  & = \frac{\gamma}{2}\int_0^{+\infty}| \phi'_{\mathfrak{b}}(t)|^2  t\, dt \\
				&+\frac{\gamma}{2}\int_0^{\infty} \left(b_2t +\left(\xi_{\mathfrak{b}}\left(1-\frac{\delta}{\gamma}\right) - \frac{b_2\delta}{2\gamma} t\right)\right)^2 |\phi_{\mathfrak{b}}(t)|^2 t\, dt  \\
				&+ \mathcal{O}(\delta^{\infty}) +  \mathcal{O}(\delta^{\frac{3}{2}}).
			\end{split}
		\end{equation*}
		Note that, using again \eqref{Decay of phi_b,l}
		\begin{equation*}
			\int_0^{+\infty} \left(b_2t +\left(\xi_{\mathfrak{b}}\left(1-\frac{\delta}{\gamma}\right) - \frac{b_2\delta}{2\gamma} t\right)\right)^2 |\phi_{\mathfrak{b}}(t)|^2 t \, dt
		\end{equation*}
		\[ =  \int_0^{+\infty} \paren*{(bt +\xi_{\mathfrak{b}})^2\abs{\phi_{\mathfrak{b}}(t)}^2  - \frac{\delta}{\gamma}(bt+\xi_{\mathfrak{b}})(bt +2\xi_{\mathfrak{b}})\abs{\phi_{\mathfrak{b}}(t)}^2} t \, dt + \mathcal{O}(\delta). \]
		Using again \eqref{Definition gamma, delta, l} we know that \(\gamma = \delta^{\frac{1}{2}}\) implying \eqref{Eq 4.3}. 
	\end{proof}
	\begin{lemma}\label{Lemma 5} It holds that, as \(\delta\rightarrow 0^+\), 
		\begin{equation}\label{Eq 4.5}
			\int_{V^+_1} |(\nabla -ib_1\mathcal{A})\Psi |^2 \, dx  = \frac{1}{2}\delta^{\frac{1}{2}} J_{b_1} + \mathcal{O}(\delta^{\frac{3}{2}}),
		\end{equation}
		with 
		\begin{equation}\label{Ja}
			\begin{split}
				J_{b_1} := &   -\int_{-\infty}^0\paren*{| \phi'_{\mathfrak{b}}(t)|^2 +(b_1t+\xi_{\mathfrak{b}})^2|\phi_{\mathfrak{b}}(t)|^2}t\, dt\\
				&-\delta^{\frac{1}{2}}\int_{-\infty}^0(b_1t + \xi_{\mathfrak{b}})(b_1t+2\xi_{\mathfrak{b}})| \phi_{\mathfrak{b}}(t)|^2 t\,  dt .
			\end{split}     
		\end{equation}
	\end{lemma}
	\begin{proof}
		In \(V^+_1\), we have \(\theta\in \paren*{\frac{3\pi+\delta }{2}, \frac{3\pi+\delta+\gamma}{2}}\) and again recall that \(\gamma\) is a small positive quantity with  \(2\delta < \gamma \) by \eqref{Definition gamma, delta, l}. Hence, 
		\[ \sin{\theta} = -1+\mathcal{O}(\delta)\text{, } \sin{2\theta} = \mathcal{O}(\delta^{\frac{1}{2}}) \text{, }\frac{1}{\sin^2{\theta}} = 1+\mathcal{O}(\delta), \]
		\begin{equation}\label{c1 and d1 expansion}
			c_1 = \xi_{\mathfrak{b}} \frac{\gamma+\delta }{2} + \mathcal{O}(\delta^{\frac{3}{2}})  , \quad  d_1  = \frac{b_1\delta}{4} + \mathcal{O}(\delta^{\frac{3}{2}}),
		\end{equation}
		and \( h'_1(\theta) = - \frac{2}{\gamma}\) in this case. The rest of the proof follows the proof of Lemma \ref{Lemma 4} above. The only difference is that \(\frac{1}{\sin^n{\theta}} = -1+\mathcal{O}(\delta)\) if \(n\in \{1,3\}\).
	\end{proof}
	
	\subsection{Energy in \(V^- \)}
	This section is analogous to the previous one and we will get similar results. However, it is important to notice that this does not follow from Lemma \ref{Lemma 1}, because \(\Psi\) is not in the form satisfying \eqref{u_- definition}. 
	\begin{lemma}\label{Lemma 6}
		We have
		\begin{multline}\label{Eq 4.14}
			\int_{V^-_{1}} |(\nabla -ib_1\mathcal{A})\Psi |^2 dx =  \int_{\frac{3\pi+\delta-\gamma }{2}}^{\frac{3\pi+ \delta}{2}} \int_0^{\frac{-\ell}{\sin{\theta}}}| \phi'_{\mathfrak{b},\ell}(r\sin{\theta})|^2 \, r\, dr\, d\theta \\
			+ \int_{\frac{3\pi+\delta-\gamma }{2}}^{\frac{3\pi+ \delta}{2}} \int_0^{\frac{-\ell}{\sin{\theta}}} F_1(r, \theta, h_1) |\phi_{\mathfrak{b},\ell}(r\sin{\theta})|^2 \, r\, dr\, d\theta ,
		\end{multline} 
		and 
		\begin{multline}\label{Eq 4.13}
			\int_{V^-_{2}} |(\nabla -ib_2\mathcal{A})\Psi |^2 dx  =\int_{\frac{\pi+\delta }{2}}^{\frac{\pi+ \delta+\gamma }{2}} \int_0^{\frac{\ell}{\sin{\theta}}}| \phi'_{\mathfrak{b},\ell}(r\sin{\theta})|^2  r\, dr \, d\theta \\
			+\int_{\frac{\pi+\delta }{2}}^{\frac{\pi+ \delta+\gamma }{2}} \int_0^{\frac{\ell}{\sin{\theta}}}F_2(r, \theta, h_2) |\phi_{\mathfrak{b},\ell}(r\sin{\theta})|^2\,  r\, dr\, d\theta,
		\end{multline}
		where \(F_j(r, \theta, h_j)\), for \(j\in \{1,2\}\), were defined in \eqref{Fj definition}. 
	\end{lemma}
	\begin{proof}
		This can be shown in the exact same way as Lemma \ref{Lemma 3}. The only difference is the integration domain in \(\theta\), but this does not add any difficulty.
	\end{proof}
	\begin{lemma}\label{Lemma 7} It holds that, as \(\delta \rightarrow 0^+\), 
		\begin{equation}\label{Eq 4.16}
			\int_{V^-_{1}} |(\nabla -i b_1 \mathcal{A})\Psi |^2 dx  = \frac{1}{2}\delta^{\frac{1}{2}}J_{b_1} + \mathcal{O}(\delta^{\frac{3}{2}}), 
		\end{equation}
		and 
		\begin{equation}\label{Eq 4.15}
			\int_{V^-_2} |(\nabla -ib_2 \mathcal{A})\Psi |^2 dx  = \frac{1}{2}\delta^{\frac{1}{2}} \delta^{\frac{1}{2}}J_{b_2} + \mathcal{O}(\delta^{\frac{3}{2}}),
		\end{equation}
		with \(J_{b_1}\) and \(J_{b_2}\) defined as in \eqref{Ja} and \eqref{J1} respectively. 
	\end{lemma}
	\begin{proof}
		This can be proven in the same way as Lemma \ref{Lemma 4} and Lemma \ref{Lemma 5} using \eqref{Definition gamma, delta, l}. If \(\theta \in \paren*{\frac{\pi+\delta}{2}, \frac{\pi+\delta +\gamma}{2}}\), we Taylor expand around \(\frac{\pi}{2}\)
		\begin{equation}
			\sin{\theta} = 1+\mathcal{O}(\delta) , \quad \sin{2\theta} = \mathcal{O}( \delta^{\frac{1}{2}}) , \quad  \frac{1}{\sin^2{\theta}} =  1+\mathcal{O}(\delta),
		\end{equation}
		and the same for \(c_2\), \(d_2\) and \(h_2(\theta)\) (see \eqref{c_2 and d_2 expansion}). If \(\theta\in \paren*{\frac{3\pi+\delta -\gamma}{2}, \frac{3\pi +\delta}{2}}\), we Taylor expand around \(\frac{3\pi}{2}\)
		\begin{equation*}
			\sin{\theta} = -1+\mathcal{O}(\delta),\quad  \sin{2\theta} = \mathcal{O}( \delta^{\frac{1}{2}}),\quad  \frac{1}{\sin^2{\theta}} = 1+\mathcal{O}(\delta),
		\end{equation*}
		and the same for  \(c_1\), \(d_1\) and \(h_1(\theta)\) as in \eqref{c1 and d1 expansion}.
	\end{proof}
	\subsection{Energy in \(T^+\) and \(T^-\)}
	In this case, thanks to the construction of  \(\Psi\), Lemma \ref{Lemma 1} can be applied, so it will be enough to compute the energy in \(T^+\). The results in this section are similar to \cite[Lemma 3.7]{Almostflat}. 
	\begin{lemma}\label{Lemma 9} It holds that, as \(\delta\rightarrow 0^+\), 
		\begin{equation*}
			\int_{T^+}\eta(x_1)^2 |(\nabla -i \sigma_{\mathfrak{b}, \delta} \mathcal{A})\Psi|^2 \, dx   \leq  \paren*{\beta_{\mathfrak{b}} + \mathcal{O}(\delta^{-\infty})} \norm{\eta}_{L^2(\R_+)}^2  +\mathcal{O}(\delta^{\frac{3}{2}}) 
		\end{equation*}
		\begin{equation*}
			- \paren*{\frac{\gamma -\delta }{2}}\int_0^{\infty} \paren*{|\phi'_{\mathfrak{b}}(t)|^2 + (b_2t + \xi_{\mathfrak{b}} )^2|\phi_{\mathfrak{b}}(t)|^2}t\, dt 
		\end{equation*}
		\begin{equation}
			+\paren*{\frac{\gamma +\delta }{2}} \int_{-\infty}^0\paren*{|\phi'_{\mathfrak{b}}(t)|^2 + (b_1t + \xi_{\mathfrak{b}})^2|\phi_{\mathfrak{b}}(t)|^2}t\, dt. 
		\end{equation}
		\begin{proof}
			Let \(\omega (x) := \phi_{\mathfrak{b},\ell}(x_2) e^{i\xi_{\mathfrak{b}} x_1}\) and \(R:= (0,+\infty) \times (-\ell,\ell)\), then 
			\[\begin{split}
				\int_{T^+}\eta(x_1)^2 |(\nabla -i\sigma_{\mathfrak{b}, \delta} \mathcal{A})\omega|^2 \,  dx  = & \int_{R}\eta(x_1)^2 |(\nabla -i\sigma_{\mathfrak{b}, \delta}  \mathcal{A})\omega|^2 \, dx\\
				& - \int_{R\setminus T^+}\eta(x_1)^2 |(\nabla -i\sigma_{\mathfrak{b}, \delta} \mathcal{A})\omega|^2 \, dx.
			\end{split} \]
			Using the same Gauge transformation as in Lemma \ref{Lemma 3}, \(\zeta(x_1, x_2) =x_1x_2/2\), and the fact that the cut-off function \(\eta(x_1) =1\) in a neighborhood of the origin, we can write
			\begin{equation*}
				\int_{R\setminus T^+}\eta(x_1)^2 |(\nabla -i\sigma_{\mathfrak{b}, \delta} \mathcal{A})\omega |^2 \, dx = E_1 + E_2  ,
			\end{equation*}
			where 
			\begin{equation*}
				E_1 :=\int_{\frac{3\pi}{2}}^{\frac{3\pi+\gamma +\delta}{2}} \int_0^{\frac{-\ell}{\sin{\theta}}}  \left(|\partial_r \omega_1|^2 +r^{-2}\bigabs{(\partial_{\theta}-ib_1\frac{r^2}{2}) \omega_1}^2 \right)\, r\, drd\theta 
			\end{equation*}
			\begin{equation*}
				E_2 :=\int_{\frac{\pi+\delta -\gamma}{2}}^{\frac{\pi}{2}} \int_0^{\frac{\ell}{\sin{\theta}}}  \paren*{|\partial_r \omega_2 |^2 +r^{-2}\bigabs{\paren*{\partial_{\theta}-ib_2\frac{r^2}{2}} \omega_2}^2}\, r\, drd\theta
			\end{equation*}
			and
			\begin{equation*}
				\omega_j(r\cos{\theta}, r\sin{\theta})= e^{i b_j\frac{r^2 \sin(2\theta)}{4}}\omega(r\cos{\theta}, r\sin{\theta})  ,
			\end{equation*}    
			where \(j\in \{1,2\}\). Note that both terms are similar except for \(b_1\) or \(b_2\), so we will compute everything for \(\omega_2\).
			\begin{equation}\label{Eq 4.18}
				\abs{\partial_r \omega_2}^2 = \sin^2{\theta} | \phi'_{\mathfrak{b},\ell}(r\sin{\theta})|^2 + \paren*{\xi_{\mathfrak{b}} \cos{\theta} + \frac{1}{2}rb_2\sin(2\theta)}^2 | \phi_{\mathfrak{b},\ell}(r\sin{\theta})|^2, 
			\end{equation}
			\begin{equation}\label{Eq 4.19}
				\begin{split}
					r^{-2}\bigabs{\paren*{\partial_{\theta}-ib_2\frac{r^2}{2}} \omega_2}^2 = & \cos^2{\theta} |\phi'_{\mathfrak{b},\ell}(r\sin{\theta})|^2 \\
					& +\paren*{ b_2 r \sin^2{\theta} + \xi_{\mathfrak{b}} \sin{\theta} }^2 |\phi_{\mathfrak{b},\ell}(r\sin{\theta})|^2,
				\end{split}
			\end{equation}
			where, as in Lemma \ref{Lemma 3}, we used the fact that \(\phi_{\mathfrak{b}}\) and \(\phi'_{\mathfrak{b}}\) are real valued. 
			Thus, combining \eqref{Eq 4.18} and \eqref{Eq 4.19}, 
			\begin{equation*}
				\begin{split}
					E_2  = &  \int_{\frac{\pi+\delta -\gamma}{2}}^{\frac{\pi}{2}} \int_0^{\frac{\ell}{\sin{\theta}}} (|\phi'_{\mathfrak{b},\ell}(r\sin{\theta})|^2 r\, drd\theta\\
					& + \int_{\frac{\pi+\delta -\gamma}{2}}^{\frac{\pi}{2}} \int_0^{\frac{\ell}{\sin{\theta}}} (b_2r\sin{\theta} +\xi_{\mathfrak{b}})^2 |\phi_{\mathfrak{b},\ell}(r\sin{\theta})|^2\,  r\, dr\, d\theta.
				\end{split}     
			\end{equation*}
			Introducing the change of variables \(t= r\sin{\theta}\) and using that \(\sin^{-2}{\theta} = 1 +\mathcal{O}(\delta) \) for \(\theta \in \paren*{\frac{\pi+\delta -\gamma}{2}, \frac{\pi}{2}}\), by \eqref{Definition gamma, delta, l} we get for all \(N\in \N\), 
			\begin{equation}\label{Eb}
				E_2 \geq   \frac{\gamma - \delta}{2} \int_{0}^{+\infty }  \paren*{|\phi'_{\mathfrak{b}}(t)|^2 + \paren*{  (b_2t +\xi_{\mathfrak{b}})^2} |\phi_{\mathfrak{b}}(t)|^2)}t\, dt -C_2\delta^{\frac{3}{2}} - C_2 \delta^{N}, 
			\end{equation}
			for some constant \(C_2>0\). For \(E_1\) we can proceed similarly and get
			\begin{equation}\label{Ea}
				E_1\geq   -\paren*{\frac{\gamma +\delta }{2}}\int_{-\infty}^0\paren*{|\phi'_{\mathfrak{b}}(t)|^2 + (b_1t + \xi_{\mathfrak{b}})^2|\phi_{\mathfrak{b}}(t)|^2}\, t\, dt -C_1\delta^{\frac{3}{2}} -C_1 \gamma \delta^N ,
			\end{equation}
			for some constant \(C_1>0\) and for all \(N\in\N\). Finally, using \eqref{Decay of phi_b,l} and \eqref{Definition gamma, delta, l}
			\begin{equation*}\label{R part}
				\begin{split}
					\MoveEqLeft[3] \int_{R}\eta(x_1)^2 |(\nabla +i\sigma_{\mathfrak{b}, \delta} \mathcal{A})\Psi|^2 \, dx \\  &=  \int_{\R_+}\eta(x_1)^2 \int_{-\ell}^{\ell}   |\phi'_{\mathfrak{b},\ell}(x_2)|^2 \, dx \\
					&\quad + \int_{\R_+}\eta(x_1)^2 \int_{-\ell}^{\ell}   (\xi_{\mathfrak{b}} +\sigma_{\mathfrak{b}, \delta} (x) x_2)^2 \abs{\phi_{\mathfrak{b},\ell }(x_2)}^2  \, dx_2 \\
					&  =  \paren*{\beta_{\mathfrak{b}} + \mathcal{O}( \delta^{\infty})} \norm{\eta}_{L^2(\R_+)}^2.
				\end{split}
			\end{equation*}
			Combining this with \eqref{Eb} and \eqref{Ea} gives the result. 
		\end{proof}
	\end{lemma}
	\subsection{Estimate of the energy of \(\Psi^{tr}\)}\label{Estimate of the energy of Psi}
	We finish this section computing the energy of the trial state \(\Psi^{tr}\). 
	\begin{lemma}\label{Lemma energy Psi (tr)} It holds that, as \(\delta\rightarrow 0^+\), 
		\begin{equation}
			\begin{split}
				\mathcal{Q}_{\mathfrak{b}, \delta} (\Psi^{tr}) =&  \int_{\Omega_{2,\delta}} |\eta(x_1)|^2 |(\nabla - ib_2 \mathcal{A}) \Psi |^2 \, dx  \\
				& + \int_{\Omega_{1, \delta}}\abs{\eta(x_1)}^2 |(\nabla - ib_1\mathcal{A})\Psi |^2 \, dx  \\
				& + 2\norm{\phi_{\mathfrak{b}}}^2_{L^2(\R)} \int_{\R_+} |\eta'_+(x_1)|^2 dx_1 + \mathcal{O}(\delta^{\infty}).
			\end{split}       
		\end{equation}
	\end{lemma}
	\begin{proof}
		Since  \(\Psi^{tr} = \eta\Psi\)
		\begin{equation}\label{Q_b, delta (Psi^tr)}
			\begin{split}
				\mathcal{Q}_{\mathfrak{b}, \delta} (\Psi^{tr})  =& \mathcal{Q}_{\mathfrak{b}, \delta} (\eta\Psi) \\
				=& \int_{\Omega_{2, \delta}} |\eta|^2 |(\nabla - ib_2 \mathcal{A}) \Psi |^2 \, dx  \\
				& + \int_{\Omega_{1, \delta}}\abs{\eta}^2 |(\nabla - ib_1\mathcal{A})\Psi |^2 \, dx\\
				& + \int_{\Omega_{2, \delta}} \abs{\Psi}^2 |\nabla \eta |^2\, dx +  \int_{\Omega_{1, \delta}}\abs{\Psi}^2 |\nabla{\eta} |^2\; dx \\
				& +\int_{\Omega_{2, \delta} } 2\Re \paren*{\bar{\Psi}\nabla \eta ( \nabla  - ib_2 \mathcal{A})\Psi} \, dx \\
				& +\int_{\Omega_{1, \delta} } 2\Re \paren*{\bar{\Psi}\nabla \eta ( \nabla  - ib_1 \mathcal{A})\Psi}\, dx .
			\end{split}        
		\end{equation}
		By construction, the normal vector \(\nu\) is orthogonal to \(\nabla \eta\). Integrating by parts,
		\[ \int_{\Omega_{j, \delta}} \abs{\Psi}^2 |\nabla \eta |^2\, dx = - \int_{\Omega_{j, \delta}} \eta\Delta \eta |\Psi|^2 dx -  \int_{\Omega_{j, \delta}} 2\Re\paren*{\eta (\nabla \eta ) (\nabla \Psi) \bar{\Psi}},\]
		for \(j\in \{1,2\}\). We can use it to simplify \eqref{Q_b, delta (Psi^tr)} if we observe 
		\[\Re (- i\abs{\Psi}^2 \nabla \eta b_2 \mathcal{A})) = 0 \text{ and } \Re (- i\abs{\Psi}^2 \nabla \eta b_1 \mathcal{A})) = 0 ,\]
		where we used the fact that \(\nabla \eta\), \(\mathcal{A}\), \(b_1\) and \(b_2\) are real valued. Hence, 
		\begin{equation*}
			\begin{split}
				\mathcal{Q}_{\mathfrak{b}, \delta} (\Psi^{tr})  = &  \int_{\Omega_{2, \delta}} |\eta|^2 |(\nabla - ib_2 \mathcal{A}) \Psi |^2 \, dx + \int_{\Omega_{1,\delta}}\abs{\eta}^2 |(\nabla - ib_1\mathcal{A})\Psi |^2 \, dx\\
				& - \int_{\Omega_{2, \delta}} \eta(x_1) \eta''(x_1)  |\Psi|^2 dx - \int_{\Omega_{1, \delta}} \eta(x_1) \eta''(x_1)  |\Psi|^2 dx .
			\end{split}
		\end{equation*}
		Observe that, by construction (see \eqref{Eta definition} and definition of \(\Psi\)), 
		\begin{multline*}
			- \int_{\Omega_{2, \delta}} \eta(x_1) \eta''(x_1)  |\Psi|^2 dx - \int_{\Omega_{1, \delta}} \eta(x_1) \eta''(x_1)  |\Psi|^2 dx      \\
			= -2\int_{T^+} \eta(x_1) \eta''(x_1)  |\phi_{\mathfrak{b}, \ell}(x_2) |^2 dx .
		\end{multline*}
		Moreover, since \(\eta'_+(x) = 0\) on a small neighbourhood of the origin,  we can restrict the last integral to \(T^+\cap(\R_+ \times \R)\). Then, using again integration by parts and the decay of \(\phi_{\mathfrak{b}, \ell}\) (see \eqref{Decay of phi_b,l}) we get the desired expression. 
	\end{proof}
	\begin{remark}
		As in \cite[Section 3.5.1]{Almostflat} the proof requires \(\eta_+\in H^2(\R)\), but a density argument gives the identity for all \(\eta_+\in H^1(\R)\) with \(\eta_+\) being 1 in a neighborhood of the origin. 
	\end{remark}
	\section{Estimate of the \(L^2\)-norm}\label{Section L^2 norm}
	In the last section we computed the energy of \(\Psi^{tr}\). In this section, we will compute the \(L^2\) norm of \(\Psi^{tr}\). 
	\begin{lemma}\label{Lemma L2 norm} It holds that, as \(\delta\rightarrow 0^+\), 
		\begin{equation}\label{L2 norm}
			\int_{\R^2}|\Psi^{tr}|^2 \, dx = (1+\mathcal{O}(\delta^{\infty}))2 \norm{\eta_+}^2_{L^2(\R_+)}+ \delta  \int_{-\infty}^{+\infty} |\phi_{\mathfrak{b}}(t)|^2 t \, dt  +\mathcal{O}(\delta^3).  
		\end{equation}
	\end{lemma}
	\begin{proof}
		From the construction of \(\Phi\), and the proof of Lemma \ref{Lemma energy Psi (tr)}, it is clear that \( \int_{T^+ \cup V^+ }|\Psi^{tr}|^2 \, dx = \int_{T^-\cup V^-}|\Psi^{tr}|^2 \, dx \) since introducing the change of variables \(y= S_{\delta}x\) we have
		\begin{equation*}
			\begin{split}
				\int_{T^-\cup V^-}|\Psi^{tr}|^2 \, dx& = \int_{T^-\cup V^-} \eta_+(-x_1\cos{\delta} - x_2\sin{\delta})^2 | \phi_{\mathfrak{b},\ell} (-x_1\sin{\delta} + x_2\cos{\delta})|^2 \,  dx \\
				& = \int_{T^+ \cup V^+} \eta_+(y_1)^2 |\phi_{\mathfrak{b},\ell}(y_2)|^2 dx =  \int_{T^+ \cup V^+}|\Psi^{tr}|^2 \, dy .
			\end{split}
		\end{equation*}
		Moreover, applying the change of variables \(t = r\sin{\theta}, dt = dr\sin{\theta}\) 
		\begin{equation*}
			\begin{split}
				\int_{T^+\cup V^+}|\Psi^{tr}|^2 dx  =  & \int_{(0, +\infty)\times(-\ell,\ell)} \eta_+(x_1)^2 |\phi_{\mathfrak{b},\ell}(x_2)|^2 \,  dx\\
				&  +\int_{\frac{\pi}{2}}^{\frac{\pi+\delta}{2}}\int_0^{\frac{\ell}{\sin{\theta}}} |\phi_{\mathfrak{b},\ell}(r\sin{\theta})|^2 r \, dr \, d\theta\\
				&  -\int_{\frac{3\pi}{2}}^{\frac{3\pi+\delta}{2}}\int_0^{\frac{-\ell}{\sin{\theta}}} |\phi_{\mathfrak{b},\ell}(r\sin{\theta})|^2 r \, dr \, d\theta \\
				= & (1+\mathcal{O}(\delta^{\infty}))\norm{\eta_+}^2_{L^2(\R_+)}+ \mathcal{O}(\delta^3)\\
				& + \frac{\delta}{2}\paren*{\int_0^{+\infty} |\phi_{\mathfrak{b}}(t)|^2 t \, dt  +\int^0_{-\infty}  |\phi_{\mathfrak{b}}(t)|^2 t \, dt } ,
			\end{split}
		\end{equation*}
		where we have used \eqref{Definition gamma, delta, l}, \eqref{Decay of phi_b,l} and \( \frac{1}{\sin^{2}{\theta}} = 1+\mathcal{O}(\delta^2)\) in the two regions considered. Due to the support of \(\Psi\), this finishes the proof. 
	\end{proof}
	\section{Proof of Theorem \ref{Main Theorem}}\label{Section proof main theorem}
	Now we will combine the previous computations to prove Theorem \ref{Main Theorem}. By Lemma \ref{Lemma energy Psi (tr)} and Lemma \ref{Lemma L2 norm} we have, as \(\delta\rightarrow 0^+\), 
	\begin{equation*}
		\begin{split}
			\mathcal{Q}_{\mathfrak{b}, \delta} (\Psi^{tr}) - \beta_{\mathfrak{b}} \norm{\Psi^{tr}}^2_{L^2(\R^2)}
			= &\int_{\R^2} |\eta|^2 |(\nabla - i\sigma_{\mathfrak{b}, \delta} \mathcal{A}) \Psi |^2 \, dx   \\
			& + 2\norm{\phi_{\mathfrak{b}}}^2_{L^2(\R)} \int_{\R_+} |\eta'(x_1)|^2 \, dx_1 \\
			& -  (1+\mathcal{O}(\delta^{\infty}))2\beta_{\mathfrak{b}} \norm{\eta_+}^2_{L^2(\R_+)} \\
			&  -  \delta \beta_{\mathfrak{b}} \int_{-\infty}^{+\infty} |\phi_{\mathfrak{b}}(t)|^2 t \, dt  +\mathcal{O}(\delta^3).  
		\end{split}
	\end{equation*}
	By Lemma \ref{Lemma 1} and the support of \(\Psi^{tr}\) we know that
	\begin{equation*}
		\begin{multlined}
			\int_{\R^2} |\eta|^2 |(\nabla - i\sigma_{\mathfrak{b}, \delta} \mathcal{A}) \Psi |^2 \, dx  \leq  2 \int_{T^+} |(\nabla - i\sigma_{\mathfrak{b}, \delta} \mathcal{A}) \Psi |^2 \, dx\\
			+ \int_{V^+_2} |(\nabla - ib_2 \mathcal{A}) \Psi |^2 \, dx  + \int_{V^-_2} |(\nabla - ib_2 \mathcal{A}) \Psi |^2 \, dx \\
			+ \int_{V^+_1} |(\nabla - ib_1 \mathcal{A}) \Psi |^2 \, dx + \int_{V^-_1} |(\nabla - ib_1 \mathcal{A}) \Psi |^2 \, dx ,
		\end{multlined}
	\end{equation*}
	where we have used that \(\abs{\eta(x)}\leq 1 \) for any \(x\in \R^2\). Thus, by Lemmas \ref{Lemma 4}, \ref{Lemma 5}, \ref{Lemma 7} and \ref{Lemma 9}, we get
	\begin{equation}\label{Eq 6.1}
		\begin{split}
			\mathcal{Q}_{\mathfrak{b}, \delta} (\Psi^{tr}) - \beta_{\mathfrak{b}} \norm{\Psi^{tr}}^2_{L^2(\R^2)}  \leq  & \, \delta \int_{0}^{\infty} \paren*{|\phi'_{\mathfrak{b}}(t)|^2 + (b_2t +\xi_{\mathfrak{b}})^2\abs{\phi_{\mathfrak{b}}(t)}^2} t\, dt\\ 
			& \, +\delta \int_{-\infty}^{0} \paren*{|\phi'_{\mathfrak{b}}(t)|^2 + (b_1t +\xi_{\mathfrak{b}})^2\abs{\phi_{\mathfrak{b}}(t)}^2} t\, dt\\
			&\, -  \delta\int_0^{\infty}  (b_2t+\xi_{\mathfrak{b}}) (b_2t+2\xi_{\mathfrak{b}}) |\phi_{\mathfrak{b}}(t)|^2t\, dt \\
			& \,-\delta\int_{-\infty}^0(b_1t + \xi_{\mathfrak{b}})(b_1t+2\xi_{\mathfrak{b}})|\phi_{\mathfrak{b}}(t)|^2t\, dt \\
			&\,   -\delta \beta_{\mathfrak{b}}  \int_{-\infty}^{+\infty} |\phi_{\mathfrak{b}}(t)|^2 \, t \, dt+2\norm{\eta'_+}^2_{L^2(\R_+)} \\
			& \, +\mathcal{O}(\delta^{\frac{3}{2}}) + \mathcal{O}(\delta^{\infty}) \norm{\eta_+}^2 _{H^1(\R_+)}, 
		\end{split}
	\end{equation}
	Note, in particular, that the two terms \(2\beta_{\mathfrak{b}}\norm{\eta_+}^2_{L^2(\R_+)}\) cancel out. To finish the proof, we will focus on the terms that have a factor \(\delta\). 
	\begin{equation}\label{J definition}
		J(\mathfrak{b}):= J_1 + J_2 +J_3,
	\end{equation}
	with 
	\begin{equation}\label{J_1}
		J_1 := - \beta_{\mathfrak{b}} \int_{-\infty}^{+\infty} |\phi_{\mathfrak{b}}(t)|^2 t dt
	\end{equation}
	\begin{equation}\label{J_2}
		\begin{split}
			J_2  := &\int_{0}^{\infty} \paren*{|\phi'_{\mathfrak{b}}(t)|^2 + (b_2t +\xi_{\mathfrak{b}})^2\abs{\phi_{\mathfrak{b}}(t)}^2} t\, dt\\
			& + \int_{-\infty}^{0} \paren*{|\phi'_{\mathfrak{b}}(t)|^2 + (b_1t +\xi_{\mathfrak{b}})^2\abs{\phi_{\mathfrak{b}}(t)}^2} t\, dt
		\end{split}
	\end{equation}
	\begin{equation}\label{J_3}
		\begin{split}
			J_3 := & -  \int_0^{+\infty} (b_2t+\xi_{\mathfrak{b}}) (b_2t+2\xi_{\mathfrak{b}}) |\phi_{\mathfrak{b}} (t)|^2 t\, dt \\
			& - \int_{-\infty}^0(b_1t + \xi_{\mathfrak{b}})(b_1t+2\xi_{\mathfrak{b}})| \phi_{\mathfrak{b}}(t)|^2 t \, dt . 
		\end{split}
	\end{equation}
	
	\begin{lemma} It holds that
		\begin{equation}\label{Eq 9.4}
			J_2 = \beta_{\mathfrak{b}} \int_{-\infty}^{\infty}    |\phi_{\mathfrak{b}}(t)|^2t\, dt  .
		\end{equation}
		In particular, \(J_1 = -J_2\). 
	\end{lemma}
	\begin{proof}
		Integrating by parts we get
		\[ \int_0^{+\infty} |\phi'_{\mathfrak{b}} (t)|^2 t dt  =\int_0^{+\infty} (t(-\phi''_{\mathfrak{b}} (t)) \phi_{\mathfrak{b}} (t)  - \phi'_{\mathfrak{b}}(t) \phi_{\mathfrak{b}}(t) )\, dt . \]
		Using that, for \(t>0\), \( -\phi''_{\mathfrak{b}} (t) = \beta_{\mathfrak{b}}  \phi_{\mathfrak{b}}(t) -(b_2t+\xi_{\mathfrak{b}})^2\phi_{\mathfrak{b}} (t)\) we obtain
		\begin{equation}\label{Eq 9.5}
			\int_0^{+\infty} |\phi'_{\mathfrak{b}}(t)|^2 t \,  dt =\int_0^{+\infty} (\beta_{\mathfrak{b}}   - (b_2t+\xi_{\mathfrak{b}})^2)|\phi_{\mathfrak{b}} (t)|^2 \, t \, dt+\frac{\phi_{\mathfrak{b}}(0)^2}{2} .
		\end{equation}
		Similarly, 
		\[\int_{-\infty}^0 |\phi_{\mathfrak{b}}'(t)|^2 t\, dt =\int_{-\infty}^0 \bigg [ t(-\phi''_{\mathfrak{b}} (t)) \phi_{\mathfrak{b}} (t)  -  \paren*{\frac{\phi_{\mathfrak{b}}(t)^2}{2}}'\bigg ] \, dt \]
		For \(t<0\) we have  \( -\phi''_{\mathfrak{b}} (t) = \beta_{\mathfrak{b}} \phi_{\mathfrak{b}}(t) -(b_1t+\xi_{\mathfrak{b}})^2\phi_{\mathfrak{b}} (t)\)
		\begin{equation}\label{Eq 9.6}
			\int_{-\infty}^0 |\phi'_{\mathfrak{b}}(t)|^2 t dt = \int_{-\infty}^0 (\beta_{\mathfrak{b}}   - (b_1t+\xi_{\mathfrak{b}})^2)|\phi_{\mathfrak{b}} (t)|^2 t dt-\frac{\phi_{\mathfrak{b}}(0)^2}{2} .
		\end{equation}
		Plugging \eqref{Eq 9.5} and \eqref{Eq 9.6} into \eqref{J_2} we get the result. 
	\end{proof}
	\begin{lemma}\label{Lemma J_3} It holds that
		\begin{equation}\label{Eq 9.7}
			J_3 =    -M_3(\mathfrak{b})+\xi_{\mathfrak{b}}^2 M_1(\mathfrak{b}) ,
		\end{equation}
		where \( M_n(\mathfrak{b})\) was defined in \eqref{Moment definition}. 
	\end{lemma}
	\begin{proof}
		Observe that for \(j\in \{1,2\}\)
		\[-t(b_jt+\xi_{\mathfrak{b}}) (b_jt+2\xi_{\mathfrak{b}}) = -\frac{1}{b_j}(b_jt+\xi)^3 + \frac{\xi_{\mathfrak{b}}^2}{b_j}\paren*{b_jt+\xi_{\mathfrak{b}}} .\]
		Hence, 
		\begin{equation*}
			\begin{split}
				J_3  = &  \int_0^{+\infty} \paren*{-\frac{1}{b_2}(b_2t+\xi)^3 + \frac{\xi_{\mathfrak{b}}^2}{b_2}\paren*{b_2t+\xi_{\mathfrak{b}}}}|\phi_{\mathfrak{b}} (t)|^2 \, dt \\ 
				& + \int_{-\infty}^0\paren*{-\frac{1}{b_1}(b_1t+\xi)^3 + \frac{\xi_{\mathfrak{b}}^2}{b_1}\paren*{b_1t+\xi_{\mathfrak{b}}}}| \phi_{\mathfrak{b}}(t)|^2 \,  dt\\
				= & -M_3(\mathfrak{b}) +\xi_{\mathfrak{b}}^2 M_1(\mathfrak{b})        \qedhere       
			\end{split}
		\end{equation*} 
		
	\end{proof}
	Collecting the previous lemmas, we see that
	\begin{equation}
		J(\mathfrak{b})= -M_3(\mathfrak{b}) +\xi_{\mathfrak{b}}^2 M_1(\mathfrak{b}) .
	\end{equation}
	Having this in mind, we rewrite \eqref{Eq 6.1} as
	\begin{equation}\label{Rayleigh quotient not simplified}
		\begin{split}
			\frac{\mathcal{Q}_{\mathfrak{b},\delta}(\Psi^{tr})}{\norm{\Psi^{tr}}^2_{L^2(\R^2)}} \leq & \beta_{\mathfrak{b}} + \frac{1}{\norm{\Psi^{tr}}^2_{L^2(\R^2)}}\bigg(\delta  J(\mathfrak{b}) +2\norm{\eta'_+}^2_{L^2(\R_+)} \\
			&+ \mathcal{O}(\delta^{\frac{3}{2}}) + \mathcal{O}(\delta^{\infty})\norm{\eta_+}^2_{H^1(\R_+)}\bigg). 
		\end{split}
	\end{equation}
	The strategy is proving that \(J(\mathfrak{b}) < 0\) for \(b_2\in (-1,0)\) and \(b_1 = 1\), and later see that the other terms in \eqref{Rayleigh quotient not simplified} are small. Using  \cite[Proposition 4.1]{Moments} we have
	\begin{enumerate}
		\item If \(b_1\in (-1, 0)\) and \(b_2=1\),
		\begin{equation*}
			M_1(\mathfrak{b}) = 0 \text{ and } M_3(\mathfrak{b})= \frac{1}{3}\underbrace{\paren*{\frac{1}{b_1}-1}}_{<0} \underbrace{\xi_{ \mathfrak{b}}}_{<0} \underbrace{\phi_{\mathfrak{b}}(0)}_{>0} \underbrace{\phi'_{\mathfrak{b}}(0)}_{<0}< 0, 
		\end{equation*}
		which means 
		\begin{equation}\label{J case b=1}
			J(\mathfrak{b}) >0,
		\end{equation}
		and this is why our construction fails in this case. 
		
		\item If \(b_1= 1\) and \(b_2\in (-1, 0)\), we need the following lemma.
		\begin{lemma}\label{Moment relation when change sign}
			\begin{equation}
				M_n(\mathfrak{b}) = (-1)^n M_n(\tilde{\mathfrak{b}}), 
			\end{equation}
			with \(\mathfrak{b} = (b_1, b_2)\) and \(\tilde{\mathfrak{b}}=(b_2, b_1)\). 
		\end{lemma}
		\begin{proof}
			Let \(s=-t\), then
			\begin{equation*}
				\begin{split}
					M_n(\mathfrak{b})& =  \int_{-\infty}^{+\infty} \frac{1}{\sigma_{\mathfrak{b}}(t)} (\xi_{\mathfrak{b}} + t)^n \abs{\phi_{\mathfrak{b}}(t)}^2 \,  dt\\
					& =\int_{-\infty}^{+\infty} \frac{1}{\sigma_{\tilde{\mathfrak{b}}}(-t)} (-\xi_{ \tilde{\mathfrak{b}}} + t)^n \abs{\phi_{\tilde{\mathfrak{b}}}(-t)}^2 \, dt \\
					& = (-1)^n \int_{-\infty}^{+\infty} \frac{1}{\sigma_{ \tilde{\mathfrak{b}}}(s)} (\xi_{ \tilde{\mathfrak{b}}} + s)^n \abs{\phi_{\tilde{\mathfrak{b}}}(s)}^2 \, ds \\
					& =  (-1)^n M_n(\tilde{\mathfrak{b}}).  \qedhere
				\end{split}
			\end{equation*}
		\end{proof}
		This lemma implies that \(M_1( \mathfrak{b}) = 0 \text{ and } M_3( \mathfrak{b})< 0\), which means 
		\begin{equation}\label{J case a=1}
			J(\mathfrak{b}) < 0,
		\end{equation}
		for \(\mathfrak{b} = (1, b_2) \) with \(b_2 \in (-1,0)\). This result is key because it is telling us in which case we will be able to go below \(\beta_{\mathfrak{b}}\) with our trial state. 
	\end{enumerate}
	Now that we know the sign of \(J(\mathfrak{b})\), we would like to make a right choice of \(\eta_+\)  to ensure the norm of \(\eta_+\) is not too large. By \eqref{L2 norm} we have
	\begin{equation*}
		\int_{\R^2}|\Psi^{tr}|^2 \, dx = (1+\mathcal{O}(\delta^{\infty}))2 \norm{\eta_+}^2_{L^2(\R_+)}+ \delta  \int_{-\infty}^{+\infty} |\phi_{\mathfrak{b}}(t)|^2 t dt  +\mathcal{O}(\delta^3),  
	\end{equation*}
	which, since \(\norm{\eta_+}_{L^2(\R_+)} \geq \epsilon\), gives
	\begin{equation*}
		\frac{1}{\norm{\Psi^{tr}}^2_{L^2(\R^2)}} = 
	\end{equation*}	
	\begin{equation}\label{Expansion 1/norm(psi)} 
		= \frac{1}{ 2\norm{\eta_+}^2_{L^2(\R_+)}}\paren*{1-\delta\frac{   \int_{-\infty}^{+\infty} |\phi_{\mathfrak{b}}(t)|^2 t dt }{\norm{\eta_+}^2_{L^2(\R_+)}} + \mathcal{O}\paren*{\frac{\delta^2}{\norm{\eta_+}_{L^2(\R_+)}^4}}+ \mathcal{O}(\delta^3)},
	\end{equation}
	and plugging this into \eqref{Rayleigh quotient not simplified} gives
	\begin{equation}\label{Rayleigh before choice of eta}
		\begin{split}
			\frac{\mathcal{Q}_{\mathfrak{b},\delta}(\Psi^{tr})}{\norm{\Psi^{tr}}^2_{L^2(\R^2)}}  \leq &  \, \beta_{\mathfrak{b}} +\frac{1}{ 2\norm{\eta_+}^2_{L^2(\R_+)}}\paren*{-\delta  M_3(\mathfrak{b}) +2\norm{\eta'_+}^2_{L^2(\R_+)}} \\
			& +  \frac{M_3(\mathfrak{b})\int_{-\infty}^{+\infty} |\phi_{\mathfrak{b}}(t)|^2 t dt }{\norm{\eta_+}^2_{L^2(\R_+)}} \delta^2 - \frac{\norm{\eta'_+}^2_{L^2(\R_+)}\int_{-\infty}^{+\infty} |\phi_{\mathfrak{b}}(t)|^2 t dt }{\norm{\eta_+}^4_{L^2(\R_+)}}\delta\\
			& + \mathcal{O}\paren*{R_{\delta}(\eta_+)},
		\end{split}
	\end{equation}
	where
	\begin{equation*}
		\begin{split}
			R_{\delta}(\eta_+)  = &  \frac{\delta^{\frac{3}{2}}}{\norm{\eta_+}^2_{L^2(\R_+)}} + \frac{\delta^2\norm{\eta'_+}^2_{L^2(\R_+)}}{\norm{\eta_+}^6_{L^2(\R_+)}} + \frac{\delta^{\frac{5}{2}}}{\norm{\eta_+}^4_{L^2(\R_+)}} \\
			& +\frac{\delta^3}{\norm{\eta_+}^6_{L^2(\R_+)}}+\frac{\delta^3 \norm{\eta'_+}^2_{L^2(\R_+)}}{\norm{\eta_+}^2_{L^2(\R_+)}} \\
			&+  \frac{\delta^{\infty}\norm{\eta_+}^2_{H^1(\R_+)}}{{\norm{\eta_+}^2_{L^2(\R_+)}}} \paren*{1+\frac{1}{\norm{\eta_+}^2_{L^2(\R_+)}}+\frac{1}{\norm{\eta_+}^4_{L^2(\R_+)}}}. 
		\end{split}  
	\end{equation*}
	
	Let \(\eta_+\) be defined as\footnote{See \cite[Section 3.5.2]{Almostflat} for a more detailed reasoning of this choice.}
	\begin{equation}\label{Choice eta}
		\eta_+(x) := \begin{cases}
			1 & \text{ if } x\in (-\epsilon, \epsilon)\\
			\exp{\frac{-M_3(\mathfrak{b})\delta}{2}(\epsilon - x)} & \text{ if } x \geq \epsilon
		\end{cases},
	\end{equation}
	which gives
	\begin{equation*}
		\norm{\eta_+}^2_{L^2(\R_+)} = \epsilon +\frac{1}{M_3(\mathfrak{b}) \delta} \text{ and } \norm{\eta'_+}^2_{L^2(\R_+)} = \frac{M_3(\mathfrak{b}) \delta}{4}. 
	\end{equation*}
	Note that with this choice
	\begin{equation}\label{Estimate delta^2 term}
		\frac{-\delta  M_3(\mathfrak{b}) +2\norm{\eta'_+}^2_{L^2(\R_+)}}{ 2\norm{\eta_+}^2_{L^2(\R_+)}} = -\frac{ M_3 (\mathfrak{b})^2 }{4}\delta^2 +\mathcal{O}(\delta^3),
	\end{equation}
	and the other terms in \eqref{Rayleigh before choice of eta} would be of higher order than \(\delta^2\). Combining \eqref{Rayleigh quotient not simplified}, \eqref{Expansion 1/norm(psi)}, \eqref{Choice eta} and  \eqref{Estimate delta^2 term}, we get
	\begin{equation}\label{Final equation}
		\frac{\mathcal{Q}_{\mathfrak{b},\delta}(\Psi^{tr})}{\norm{\Psi^{tr}}^2_{L^2(\R^2)}} \leq \beta_{\mathfrak{b}} - \delta^2\frac{M_3(\mathfrak{b})^2}{4} + o(\delta^2), 
	\end{equation}
	which finishes the proof.
	\section{Applications}\label{Section applications}
	In this section we will present a setting where \(\mathcal{L}_{\mathfrak{b}, \delta}\) can be applied as an effective model operator. We will borrow some notation from~\cite{FournaisHelfferbook}. Let \(\Omega \subset \R^2\) be an open, bounded, simply connected and smooth set. Moreover, assume that \(\Omega = \text{int} \paren{\overline{\Omega}_1 \cup \overline{\Omega}_2}\), where \(\Omega_1\) and \(\Omega_2\) are open sets such that \(\Omega_1\cap \Omega_2\) can be described as
	\begin{equation}\label{Parametrization}
		\bigg \{(x_1, x_2) \in (-L_- , 0) \times \R_- : x_1 = \tan(\frac{\pi}{2}-\delta) x_2 \bigg \} \cup \paren*{(0, L_+) \times \{0\}}, 
	\end{equation}
	where \(L_-\) and \( L_+\) are positive constants. A typical such example is illustrated in Figure \ref{Figure3}.
	
	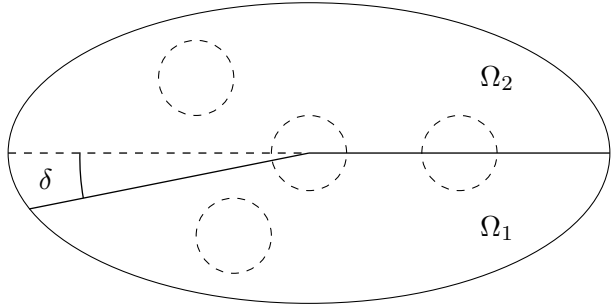
\begin{figure}[ht!]
		\begin{tikzpicture}
			\draw (0,0) ellipse (4cm and 2cm);
			\draw [line width=0.5pt, domain=0:4] plot(\x,0); 
			\draw [dashed, line width=0.5pt, domain=-4:0] plot(\x,0); 
			\draw [line width=0.5pt, domain=-3.72:0] plot(\x,\x*0.2); 
			\draw [line width=0.5pt] (-3,-0.6) arc (190: 178.5: 3.05); 
			\draw[color= black] (-3.5, -0.35) node {\(\delta\)}; 
			\draw[color= black] (2.5, -1) node {\(\Omega_1\)}; 
			\draw[color= black] (2.5, 1) node {\(\Omega_2\)}; 
			\draw[dashed, color= black] (0,0) circle (0.5cm); 
			\draw[dashed, color= black] (-1,-1.1) circle (0.5cm); 
			\draw[dashed, color= black] (-1.5,1) circle (0.5cm); 
			\draw[dashed, color= black] (2,0) circle (0.5cm); 
		\end{tikzpicture}
		\caption{Example of a domain fulfilling the assumptions stated before. We also show the four type of region where we will localize the operator.}\label{Figure3}
	\end{figure}
	In order to apply Theorem \ref{Main Theorem} so that \(\lambda_{\mathfrak{b}}(\delta)\) is an eigenvalue for \(\delta\) small enough, we let \(\mathfrak{b} = (1, b)\) with \(b \in(-1, 0)\), and \(\mathcal{A} = \mathcal{A}_{\mathfrak{b}, \delta}\), where \(\mathcal{A}_{\mathfrak{b}, \delta}\) was defined in \eqref{Aa,b,delta definition}. Consider the quadratic form
	\begin{equation}\label{Q^D_B,A definition}
		\mathcal{Q}^D_{B\mathcal{A}, \Omega}(u):= \int_{\Omega} \abs{(\nabla -i B\mathcal{A})u}^2\, dx
	\end{equation}
	with domain \(\mathcal{D}( \mathcal{Q}^D_{B\mathcal{A}, \Omega}) = H_0^1 (\Omega)\) and where \(B\) is a parameter measuring the strength of the magnetic field. The superscript \(D\) is written because what we are studying the Dirichlet realization of the Schr\"odinger operator 
	\begin{equation}
		P_{B\mathcal{A}, \Omega}  := -(\nabla -iB\mathcal{A})^2 
	\end{equation}
	in \(\Omega\). Note that \(\mathcal{Q}^D_{B\mathcal{A}, \Omega}\) is a positive quadratic form, and \(H_0^1(\Omega) = \overline{C_c^{\infty}(\Omega)}^{\norm{\cdot}_{\mathcal{V}^D}} \) with
	\begin{equation*}
		\norm{u}_{\mathcal{V}^D}^2 =   q^D_{B\mathcal{A}, \Omega}(u,u)+\norm{u}^2_{L^2(\Omega)}, 
	\end{equation*}
	where 
	\begin{equation*}
		q^D_{B\mathcal{A}, \Omega}(u,v) := \int_{\Omega} \overline{(\nabla -iB \mathcal{A})u} (\nabla -iB \mathcal{A}) v \, dx , 
	\end{equation*}
	is the sesquilinear form with domain \(H_0^1(\Omega) \times H_0^1(\Omega)\). Due to these reasons, if we consider \(\Tilde{ P}_{B\mathcal{A}, \Omega} u = P_{B\mathcal{A}, \Omega}u\)  for \(u\in \mathcal{D}(\Tilde{ P}_{B\mathcal{A}, \Omega}) = C_c^{\infty}(\Omega)\),  we can consider its Friedrichs extension, \(P^D_{B\mathcal{A}, \Omega}\), which is obtained using Lax--Milgram lemma and Riesz's theorem. We can identify \(P^D_{B\mathcal{A}, \Omega} u\) as 
	\begin{equation}
		\langle P^D_{B\mathcal{A}, \Omega}u, v\rangle_{L^2(\Omega)} =  q^D_{B\mathcal{A}, \Omega}(u, v), \, \forall v \in C_c^{\infty}(\Omega) ,
	\end{equation}
	and this means
	\begin{equation}
		\mathcal{D}(P^D_{B\mathcal{A}, \Omega}) := \{u\in H_0^1(\Omega) : P^D_{B\mathcal{A}, \Omega}u \in L^2(\Omega)\}. 
	\end{equation}
	Recall that for \(u\in  \mathcal{D}(P^D_{B\mathcal{A}, \Omega})\), 
	\begin{equation*}
		P^D_{B \mathcal{A}, \Omega} u = P_{B \mathcal{A}, \Omega}u .
	\end{equation*}
	As usual, one can show that the spectrum of \(P^D_{B \mathcal{A}, \Omega}\) is discrete using the compactness of the inclusion \(H_0^1(\Omega) \hookrightarrow L^2(\Omega)\). We denote by \(\lambda_1^D (B)\) the lowest eigenvalue of  \(P^D_{B \mathcal{A}, \Omega}\). The goal is to study the asymptotics of \(\lambda_1(B)\) when \(B\rightarrow +\infty\). 
	
	Using a change of gauge (combine \eqref{First change of Gauge} and \(\zeta(x_1, x_2) = \frac{x_1x_2}{2}\)) we can reduce the study to the same operator with magnetic potential \(\mathcal{A}_0(x_1, x_2)= \frac{1}{2}(-x_2, x_1)\). Then one can show in a similar way as in \cite[Proposition 3.5]{Bonnthesis} that the bottom of the spectrum of  \(P_{B\mathcal{A}, \Omega}\) remains unchanged if we apply a rotation to \(\Omega\). Because of this, if we rotate \(\Omega\) the statement of the following theorem does not change. 
	
	\begin{theorem}\label{Theorem 2}
		If \(b\in (-1, 0)\), then there exists a \(\delta_0 >0\) such that for all \(\delta<\delta_0\), we can find constants \(C_1, C_2, C_3>0\) such that the smallest eigenvalue of the Dirichlet realization  \(P^D_{B\mathcal{A}, \Omega}\) of  \(P_{B\mathcal{A}, \Omega}\) satisfies as  \(B\rightarrow +\infty\) 
		\begin{equation}\label{Theorem 2 equation}
			-C_1 B^{\rho_1} \leq \lambda_1(B) -  B \lambda_{\mathfrak{b}}(\delta) \leq C_2 e^{-C_3 B^{\rho_2}},
		\end{equation}
		where \(\lambda_{\mathfrak{b}}(\delta)\) was defined just before Theorem \ref{Main Theorem}, \(\rho_1\in(0,1)\) and \(\rho_2\in \paren*{0,\frac{1}{2}}\).  
	\end{theorem}
	\begin{remark}
		A more general case can be considered by having smooth curves instead of straight lines splitting \(\Omega\) into \(\Omega_1\) and \(\Omega_2\). This case can be treated similarly (obtaining slightly worse error terms), but the proof becomes more technical since we need to introduce Frenet coordinates. Examples of this approach can be found, for example, in \cite{Bonnthesis, FournaisHelfferbook, AssaadBreakdown}. 
		
		Another technical issue of considering a more general case is that magnetic potential cannot be stated explicitly. However, one can use \cite[Lemma A.1]{AssadInfluence}, \cite[Appendix D]{FournaisHelfferbook} and \cite[Appendix C]{AssaadBreakdown} to show the existence of a magnetic potential \(\mathcal{A}\in H^1_{\text{div}}(\Omega)\) satisfying \(\text{curl }\mathcal{A} = \mathbbm{1}_{\Omega_{1} }+ b\mathbbm{1}_{\Omega_{2}}\) and that the domains of \(  \mathcal{Q}^D_{B \mathcal{A}, \Omega}\) and \( P^D_{B \mathcal{A}, \Omega}\) do not depend on \(b\). This would be useful to prove Theorem \ref{Monotonicity theorem} in the more general case. 
		
		Several almost flat corners could also be introduced in the magnetic barrier leading to a competition between them in the computation of the ground state energy asymptotics. Since we have no explicit control of how the ground state energy \(\lambda_{\mathfrak{b}}(\delta)\) depends on \(\delta\) (see Theorem \ref{Main Theorem}), it will be hard to provide a satisfactory result. 
		
		We could also consider be to consider the Neumann realization of \(P_{B\mathcal{A}, \Omega}\). Then there will be a competition coming from the modal operator discussed in \cite{AssaadBreakdown} and from \(\mathcal{L}_{\mathfrak{b}, \delta}\). The lack of lower bound for both modal operators makes it difficult to evaluate where the eigenfunctions will localize.
	\end{remark}
	\begin{remark}
		It is worth to mention the comparison between this new operator considered as a model operator in Theorem \ref{Theorem 2} and \cite[Theorem 1.2]{Semiclassicalbrokenedge}, where a semi-classical eigenvalue estimate was computed. In both cases we see how \(M_3(\mathfrak{b})\) appears in the first correction term of the asymptotic expansion. Similarly, one can see a clear connection between \cite[Theorem 1.2]{Semiclassicalbrokenedge} and the case where the model operator is the one coming from the slightly curved magnetic barrier stated in Remark \ref{Slight curve remark}. Moreover, we stress that the observation made in \cite[Remark 1.3]{Almostflat} applies here too. In other words, Theorem \ref{Main Theorem} can be seen as a formal consequence of this slightly curved half-plane model. 
	\end{remark}
	We will prove Theorem \ref{Theorem 2} by finding an upper and a lower bound.
	\subsection{Agmon estimate}
	Before starting with the proof of Theorem \ref{Theorem 2} we need a result regarding the exponential decay of the ground state \(u_{\mathfrak{b}, \delta}\) of \(\mathcal{L}_{\mathfrak{b}, \delta}\) far away from the almost flat angle for \(\mathfrak{b} = (1,b)\) with \(b\in (-1,0)\), and with \(\delta\) small enough such that Theorem \ref{Main Theorem} holds. The Agmon estimate we will obtain is similar to \cite[Lemma 3.9]{Bonnthesis} and \cite[Proposition 4.9]{Raymondbook}. First, we recall a version of the IMS localization formula which can be applied in this context.
	\begin{lemma}[IMS localization formula]\label{IMS formula 1}
		Let \((\chi_{j})_{j\in J} \subset C^{\infty}(\R^2)\) be a partition of unity with \(D^{\alpha}\chi^{\alpha}_{j} \in L^{\infty}(\R^2)\) for \(\abs{\alpha}\leq 2\). Then 
		\begin{equation}\label{IMS formula version}
			\mathcal{Q}_{\mathcal{A}, \Omega}(v) = \sum_{j\in J}   \mathcal{Q}_{\mathcal{A}, \Omega}(\chi_j v) - \sum_{j\in J} \norm{\nabla \chi_j v}^2_{L^2(\R^2)}
		\end{equation}
	\end{lemma}
	\begin{proof}
		The proof goes similarly as \cite[Proposition 4.2]{Raymondbook}.
	\end{proof}
	Now we are ready to state the theorem.
	\begin{theorem}\label{Exponential decay eigenfunctions}
		Let \(\mathfrak{b} =  (1, b)\) with \(b\in (-1,0)\) and \(\delta_0 \in (0, \pi)\) such that Theorem \ref{Main Theorem} holds for \(\mathcal{L}_{\mathfrak{b}, \delta}\) with \(\delta < \delta_0\). Consider \(u_{\mathfrak{b}, \delta}\) the corresponding normalized ground state. Then for all \(\epsilon \in (0, \sqrt{\beta_{\mathfrak{b}} - \lambda_{\mathfrak{b}}(\delta)})\), there exists a constant \(C_{\epsilon, \delta}\) such that 
		\begin{equation}\label{Agmon estimate equation}
			\norm{u_{\mathfrak{b}, \delta} e^{\epsilon \abs{\cdot}}}^2_{L^2(\R^2)} + \mathcal{Q}_{\mathfrak{b},\delta} (u_{\mathfrak{b},\delta} e^{\epsilon \abs{\cdot}}) \leq C_{\epsilon, \delta},
		\end{equation}
		where \(\mathcal{Q}_{\mathfrak{b},\delta}\) was defined in \eqref{Q_a,b,delta}.
	\end{theorem}
	To ease the read, we lifted some of the technical, but straightforward, calculations to Appendix \ref{AppendixA}.
	\begin{proof}
		This proof is similar to \cite[Proposition 4.9]{Raymondbook} and \cite[Theorem 9.1]{Bonnthesis}. For \(m>0\), let
		\begin{equation*}
			\chi_m (t) :=\begin{cases}
				t & \text{ for } 0\leq x \leq m \\
				2m -t & \text{ for } m \leq x \leq 2m\\
				0 & \text{ otherwise}
			\end{cases}
		\end{equation*}
		then \(\abs{\chi_m (t)}\leq  1\) for all \(t\in \R\). Define \(f_m(x):= \epsilon \chi_m(\abs{x})\), then 
		\begin{equation*}
			e^{f_m}u_{\mathfrak{b}, \delta } \in \mathcal{D}(\mathcal{Q}_{\mathfrak{b},\delta})\text{ and } e^{2f_m}u_{\mathfrak{b}, \delta } \in \mathcal{D}(\mathcal{L}_{\mathfrak{b},\delta}).
		\end{equation*}
		Note that \(f_m\) is uniformly Lipschitz and has compact support. Integrating by parts and using \eqref{La} we get
		\begin{equation*}
			\begin{split}
				\langle \mathcal{L}_{\mathfrak{b},\delta} u_{\mathfrak{b}, \delta} , e^{2f_m} u_{\mathfrak{b}, \delta}  \rangle_{L^2(\R^2)}   & = \sum_{k=1}^2 \langle (\nabla -i A_k)^2 u_{\mathfrak{b}, \delta} , e^{2f_m}u_{\mathfrak{b}, \delta}  \rangle_{L^2(\R^2) }\\
				&  =\sum_{k=1}^2 \langle (\nabla -i A_k) u_{\mathfrak{b}, \delta} , (\nabla -i A_k) (e^{2f_m }u_{\mathfrak{b}, \delta} ) \rangle_{L^2(\R^2) } .       
			\end{split}  
		\end{equation*}    
		Let  \(P_k =(\nabla -i A_k)  \) for \(k\in \{1,2\}\) , then 
		\[P_k (e^{2f_m }u_{\mathfrak{b}, \delta}  ) = e^{f_m}[P_k, e^{f_m}]u_{\mathfrak{b}, \delta}  + e^{f_m} P_k (e^{f_m} u_{\mathfrak{b}, \delta} ),\] 
		where \(k\in \{1,2\}\) (we have followed the notation in \cite{Raymondbook}). Thus, 
		\begin{equation}\label{Integration by parts Agmon}
			\begin{split}
				\langle \mathcal{L}_{\mathfrak{b},\delta} u_{\mathfrak{b}, \delta} , e^{2f_m} u_{\mathfrak{b}, \delta}  \rangle_{L^2(\R^2)}   ={}&  \sum_{k=1}^2\langle e^{f_m} P_k  u_{\mathfrak{b}, \delta}  , [P_k, e^{f_m}]u_{\mathfrak{b}, \delta} +P_k (e^{f_m} u_{\mathfrak{b}, \delta} ) \rangle_{L^2(\R^2)} \\
				={}&  \sum_{k=1}^2 \bigg(\langle e^{f_m   } P_k u_{\mathfrak{b}, \delta} , [P_k, e^{f_m}]u_{\mathfrak{b}, \delta} \rangle_{L^2(\R^2)} \\
				& +\langle P_k (e^{f_m}u_{\mathfrak{b}, \delta}),  P_k (e^{f_m}u_{\mathfrak{b}, \delta}) \rangle_{L^2(\R^2)}\\
				& +\langle [e^{f_m}, P_k]u_{\mathfrak{b}, \delta},  P_k (e^{f_m} u_{\mathfrak{b}, \delta} )\rangle_{L^2(\R^2)}\bigg)\\
				={}&  \mathcal{Q}_{\mathfrak{b},\delta} (e^{f_m }u_{\mathfrak{b}, \delta}) - \norm{\abs{\nabla f_m}e^{f_m}u_{\mathfrak{b}, \delta}}^2_{L^2(\R^2)} \\
				& +2\sum_{k=1}^2 \Im{\langle e^{f_m}P_k u_{\mathfrak{b}, \delta} , \partial_k(e^{f_m}) u_{\mathfrak{b}, \delta} \rangle_{L^2(\R^2)}}, 
			\end{split}
		\end{equation}
		where we have used \( e^{f_m} P_k  u_{\mathfrak{b}, \delta}  =P_k  (e^{f_m} u_{\mathfrak{b}, \delta})+  [e^{f_m}, P_k]u_{\mathfrak{b}, \delta} \). Taking real part on both sides in the previous equation (note that \(f_m\) is a real-valued function) we can write 
		\begin{equation}\label{First bound quadratic form exp decay}
			\mathcal{Q}_{\mathfrak{b}, \delta} (e^{f_m }u_{\mathfrak{b},\delta}) = \lambda_{\mathfrak{b}}(\delta)  \norm{e^{f_m }u_{\mathfrak{b},\delta}}^2_{L^2(\R^2) } + \norm{\abs{\nabla f_m}e^{f_m}u_{\mathfrak{b},\delta}}^2_{L^2(\R^2)}. 
		\end{equation}
		By construction, \(f_m\) is a radial function and \(\abs{\chi'_m(x) }\leq 1\), thus
		\begin{equation*}
			\norm{\abs{\nabla f_m}e^{f_m}u_{\mathfrak{b}, \delta}}^2_{L^2(\R^2)} \leq \epsilon^2\norm{e^{f_m}u }^2_{L^2(\R^2)}. 
		\end{equation*}
		Plugging this into \eqref{First bound quadratic form exp decay} we get
		\begin{equation}\label{Quadratic form decay}
			\mathcal{Q}_{\mathfrak{b},\delta} (e^{f_m }u) \leq \paren*{\lambda_{\mathfrak{b}}(\delta) + \epsilon^2} \norm{e^{f_m}u }^2_{L^2(\R^2)}. 
		\end{equation}
		The next step is introducing a suitable partition of unity and use the IMS formula stated in Lemma \ref{IMS formula 1}. By Lemma \ref{Lemma appendix}, we know that for any \(\epsilon \in \big (0, \sqrt{\beta_{\mathfrak{b}} - \lambda_{\mathfrak{b}}(\delta)}\big )\) we can find a \(R_{\epsilon} >0\) such that \(\Sigma_{R_{\epsilon}} > \lambda_{\mathfrak{b}}(\delta)+\epsilon^2\) where \(\Sigma_{R_{\epsilon}}\) was defined in \eqref{Sigma_R}.    
		Let \(\chi_{j,R_{\epsilon}}\) with \(j\in \{1,2\}\) be a partition of unity such that \(\text{supp }\chi_{2,R_{\epsilon}} \subset \R^2 \setminus D(0,R)\) for some \(R\geq R_{\epsilon}\), and \(\sum_{j=1}^2 \abs{\nabla \chi_{j,R}}\leq CR^{-2}\). By the IMS formula (see \eqref{IMS formula version})
		\begin{equation}\label{Exponential decay first after IMS}
			\mathcal{Q}_{\mathfrak{b},\delta} (e^{f_m}u_{\mathfrak{b}, \delta})  \geq \mathcal{Q}_{\mathfrak{b},\delta} (\chi_{1, R_{\epsilon}} e^{f_m}u_{\mathfrak{b}, \delta}) +\mathcal{Q}_{\mathfrak{b},\delta} (\chi_{2, R_{\epsilon}}e^{f_m}u_{\mathfrak{b}, \delta}) - \frac{C}{R^2}\norm{e^{f_m}u_{\mathfrak{b}, \delta}}^2_{L^2(\R^2)}.
		\end{equation}
		By \eqref{First bound quadratic form exp decay} we can rewrite the previous expression as
		\begin{equation*}
			\begin{split}
				\lambda_{\mathfrak{b}}(\delta)  \norm{e^{f_m }u_{\mathfrak{b},\delta}}^2_{L^2(\R^2) } + \norm{\abs{\nabla f_m}e^{f_m}u_{\mathfrak{b},\delta}}^2_{L^2(\R^2)} \geq &  \, \sum_{j=1}^2 \mathcal{Q}_{\mathfrak{b},\delta} (\chi_{i, R_{\epsilon}} e^{f_m}u_{\mathfrak{b}, \delta})\\
				&- \frac{C}{R^2}\norm{e^{f_m}u_{\mathfrak{b}, \delta}}^2_{L^2(\R^2)}. 
			\end{split}
		\end{equation*}
		Noting that \(\abs{\nabla f_m} \leq \epsilon\) a.e., we have
		\begin{equation}\label{Mid computation Agmon estimate}
			\begin{split}
				\paren{\lambda_{\mathfrak{b}}(\delta)+\epsilon^2}  \norm{e^{f_m }u_{\mathfrak{b},\delta}}^2_{L^2(\R^2) } \geq & \,   \sum_{j=1}^2 \mathcal{Q}_{\mathfrak{b},\delta} (\chi_{i, R_{\epsilon}} e^{f_m}u_{\mathfrak{b}, \delta}) \\
				& - \frac{C}{R^2}\norm{e^{f_m}u_{\mathfrak{b}, \delta}}^2_{L^2(\R^2)}. 
			\end{split}
		\end{equation}
		Since \(\chi_{2, R_{\epsilon}} e^{f_m}u_{\mathfrak{b}, \delta} \in \mathcal{M}_R\), where \(\mathcal{M}_R\) was defined in \(\eqref{M_R}\), then Lemma \ref{Lemma appendix} gives
		\begin{equation*}
			\mathcal{Q}_{\mathfrak{b},\delta} (\chi_{2, R_{\epsilon}} e^{f_m}u_{\mathfrak{b}, \delta}) \geq \paren*{\beta_{\mathfrak{b}} - \frac{c}{R^2}}\norm{\chi_{2, R_{\epsilon}} e^{f_m}u_{\mathfrak{b}, \delta} }
		\end{equation*}
		for some constant \(c>0\). Taking this into account into \eqref{Mid computation Agmon estimate} together with the fact that \( \mathcal{Q}_{\mathfrak{b},\delta} (\chi_{1, R_{\epsilon}} e^{f_m}u_{\mathfrak{b}, \delta})\geq 0\), we have
		\begin{multline}
			\paren*{\beta_{\mathfrak{b}} - \lambda_{\mathfrak{b}}(\delta) - \epsilon^2 -\frac{\paren*{c+C}}{R^2} } \norm{\chi_{2, R_{\epsilon}} e^{f_m}u_{\mathfrak{b}, \delta}}^2\\ \leq \paren*{\lambda_{\mathfrak{b}}(\delta)+\epsilon^2 + \frac{c}{R^2}} \norm{\chi_{1, R_{\epsilon}} e^{f_m}u_{\mathfrak{b}, \delta}}^2 ,
		\end{multline}
		and for \(R \geq R_{\epsilon}\) large enough we can guarantee that the first factor is positive, and then find a positive constant \(C_2(R, \epsilon)>0\) such that
		\begin{equation}\label{Bound for chi2}
			\norm{\chi_{2, R_{\epsilon}} e^{f_m}u_{\mathfrak{b}, \delta}}^2\leq C_2(R_0, \epsilon) \norm{u_{\mathfrak{b}, \delta}}^2 = C_2(R_0, \epsilon), 
		\end{equation}
		where we have used 
		\begin{equation} \label{Bound for chi1}
			\norm{\chi_{1, R_{\epsilon}} e^{f_m}u_{\mathfrak{b}, \delta}}^2 \leq e^{2\epsilon R_0} \int_{D(0,R)} \abs{u_{\mathfrak{b}, \delta}}^2 \, dx \leq C_1(R_0,\epsilon). 
		\end{equation}
		Summing up \eqref{Bound for chi2} and \eqref{Bound for chi1} we get
		\begin{equation}\label{Final agmon estimate}
			\norm{e^{f_m}u_{\mathfrak{b}, \delta}}^2 \leq C(R_0, \epsilon),
		\end{equation}
		where \(C(R_0, \epsilon) = C_1(R_0, \epsilon) +C_2(R_0, \epsilon)\). Note that \eqref{Final agmon estimate} holds for any \(m\geq 1\), thus letting \(m\rightarrow +\infty\) and using Fatou's lemma we get the result for the \(L^2\) norm of \(e^{\epsilon \abs{\cdot}} u_{\mathfrak{b}, \delta} \). For \(\mathcal{Q}_{\mathfrak{b},\delta} (e^{\epsilon \abs{\cdot} }u)\) we can use \eqref{Quadratic form decay} and \eqref{Final agmon estimate}. 
	\end{proof}
	\subsection{Upper bound in Theorem \ref{Theorem 2}}
	We can use the same ideas as in \cite[Section 1.4]{FournaisHelfferbook} and \cite[Section 4]{AssaadBreakdown}.
	\begin{prop}\label{Proposition upper bound}
		Let \(\mathfrak{b} = (1,b)\) with \(b\in (-1,0)\), then there exists \(\delta_0>0\) such that for all \(\delta<\delta_0\) there exists constants \(C_2, C_3, B_2 >0\) such that for all \(B>B_2\)
		\begin{equation}\label{Upper bound}
			\lambda_1(B) - B\lambda_{\mathfrak{b}}(\delta)  \leq C_2 e^{-C_3B^{\rho_2}},
		\end{equation}
		for \(\rho_2 \in \paren*{0, \frac{1}{2}}\). 
	\end{prop}
	\begin{proof}
		We denote by \(D(c, r)\) the open disc in \(\R^2\) of center \(c\) and radius \(r\). Let \(\chi \in C^{\infty}(\R^2)\) be a cut-off function such that \( 0\leq \chi \leq 1 \),
		\begin{equation*}
			\chi(x) = \begin{cases}
				1 &  \text{ in } D\paren*{0, \frac{1}{2}} \\
				0 & \text{ if } x\in \R^2 \setminus D(0,1)
			\end{cases}, 
		\end{equation*}
		and \( \abs{\nabla \chi (x)}\leq C \text{ for all } x\in \R^2,\) where \(C>0\) is some positive constant. As we mentioned before, we will work under the parametrization stated in \eqref{Parametrization}. For \(B>0\) large enough we have that
		\begin{equation*}
			D(0, B^{-\rho_2}) \subset \Omega.
		\end{equation*}
		Define \(\chi_{\rho_2}(x) := \chi(B^{\rho_2}x)\), then \(\chi_{\rho_2}\) is a smooth function such that \( 0\leq \chi_{\rho_2} \leq 1 \), and
		\begin{equation}\label{chi tilde}
			\chi_{\rho_2}(x) = \begin{cases}
				1 &  \text{ in } D\paren*{0, \frac{1}{2}B^{-\rho_2}} \\
				0 & \text{ if } x \in \R^2 \setminus D(0,B^{-\rho_2})
			\end{cases}, 
		\end{equation}
		and \(\abs{\nabla \chi_{\rho_2}(x)}\leq C B^{\rho_2}\) for all \(x\in \R^2\). Theorem \ref{Main Theorem} allows us to find a \(\delta_0\) such that for all \(\delta\in (0, \delta_0)\) we can find a normalized eigenfunction \(u_{\mathfrak{b}, \delta}\)  corresponding to \(\lambda_{\mathfrak{b}}(\delta)\). Let 
		\begin{equation}\label{Trial upper bound}
			u(x) :=  u_{\mathfrak{b}, \delta}(\sqrt{B} x) \chi_{\rho_2}(x),.
		\end{equation}
		We have that \(u\in \mathcal{D}(\mathcal{Q}^D_{B\mathcal{A}, \Omega})\) and
		\begin{equation}\label{Upper bound Q energy}
			\begin{split}
				\mathcal{Q}^D_{B\mathcal{A}, \Omega}(u) \leq &  \, B^{\frac{1}{2}} \int_{D(0, B^{\frac{1}{2}-\rho_2})} \abs{(\nabla_y -i \mathcal{A}(y)) u_{\mathfrak{b}, \delta}(y)}^2\, dy \\
				& \, + CB^{ \rho_2 -\frac{1}{2}} \int_{D(\frac{1}{2}B^{\frac{1}{2}-\rho_2}, B^{\frac{1}{2}-\rho_2})} \abs{ u_{\mathfrak{b}, \delta} (y)}^2 \, dy \\
				\leq &\,  B^{\frac{1}{2}} \lambda_{\mathfrak{b}}(\delta)+  C'B^{ \rho_2 -\frac{1}{2}} e^{-\frac{C_2 B^{\frac{1}{2}-\rho_2}}{2}} , 
			\end{split}       
		\end{equation}
		where \(y=\sqrt{B}x\) and we have used the properties of \( \chi_{\rho_2}\) together with Theorem \ref{Exponential decay eigenfunctions}. Moreover, 
		\begin{equation}\label{Upper bound norm}
			\begin{split}
				\norm{u}_{L^2(\Omega)}^2 \geq &   B^{-\frac{1}{2}}  \int_{D(0, \frac{1}{2}B^{\frac{1}{2}-\rho_2})} \abs{ u_{\mathfrak{b}, \delta}(y)}^2 \, dy \\
				\geq &  B^{-\frac{1}{2}} \big (1 - C'' e^{-\frac{1}{2}C_2 B^{\frac{1}{2}-\rho_2}}\big ) ,
			\end{split}
		\end{equation}
		where we have used again the exponential decay shown in Theorem \ref{Exponential decay eigenfunctions}. By the min-max principle we have
		\begin{equation*}
			\lambda_1(B) \leq \frac{\mathcal{Q}^D_{B\mathcal{A}, \Omega}(u)}{ \norm{u}_{L^2(\Omega)}^2} \leq B \lambda_{\mathfrak{b}}(\delta) + \mathcal{O}\paren{e^{-B^{\frac{1}{2}- \rho_2}}} \qedhere
		\end{equation*}
	\end{proof}
	
	\subsection{Lower bound in Theorem \ref{Theorem 2}}
	To get the lower bound, we will use a localization argument. We follow the ideas from \cite[Section 8.2.2]{FournaisHelfferbook}.
	\begin{prop}\label{Lower bound proposition}
		Let \(\mathfrak{b} = (1,b)\) with \(b\in (-1,0)\) and \(0<\rho_1<1\), then there exists constants \(C_1, B_1 >0\) such that for all \(B>B_1\)
		\begin{equation}\label{Lower bound}
			-C_1B^{\rho_1} \leq  \lambda_1(B) - B\lambda_{\mathfrak{b}}(\delta).
		\end{equation}
	\end{prop}
	\begin{proof}
		Consider a partition of unity \(\{\chi_j^B\}_{j\in J}\) of \(\Omega\) such that there exists constants \(C, R_0>0\) satisfying
		\begin{equation}
			\sum_{j\in J} \abs{\chi_{j}^B}^2 = 1,
		\end{equation}
		\begin{equation}\label{Partition of unity derivative bound}
			\sum_{j\in J} \abs{\nabla \chi_{j}^B}\leq C R_0^{-2} B^{2\rho_1}, 
		\end{equation}
		\begin{equation}\label{Support chi_1^B}
			\text{supp } \chi_1^B \subset D(0, R_0 B^{-\rho_1})
		\end{equation}
		and, for \(j\neq 1\), 
		\begin{equation}\label{Avoid origing condition}
			\text{supp } (\chi_j^B) \subset D(z_j, R_0 B^{-\rho_1}) \text{ and } (0,0) \notin  \text{supp } (\chi_j^B)
		\end{equation}
		with
		\begin{equation}\label{Two cases partition}
			\text{either supp } \chi_j^B \cap \Gamma =  \varnothing , \text{ or } z_j \in \Gamma , 
		\end{equation}
		where \(\Gamma := (\partial \Omega_1 \cup \partial \Omega_2) \setminus \partial \Omega\). By IMS formula (see Lemma \ref{IMS formula 1}), for all \(u\in H^1_0(\Omega)\) 
		\begin{equation}
			\mathcal{Q}_{\mathcal{A}, \Omega}(u) =  \mathcal{Q}_{\mathcal{A}, \Omega}(\chi_1^B u)+ \sum_{\text{int}}  \mathcal{Q}_{\mathcal{A}, \Omega}(\chi_j^B u) + \sum_{\Gamma}  \mathcal{Q}_{\mathcal{A}, \Omega}(\chi_j^B u)- \sum_j  \norm{\abs{\nabla \chi_j^B}u}^2_{L^2(\Omega)},
		\end{equation}
		where \(\sum_{\text{int}}\) refers to the first case in \eqref{Two cases partition} and \(\sum_{\Gamma}\) to the second. Note that the last term can be bounded using \eqref{Partition of unity derivative bound} by
		\begin{equation}\label{Last term lower bound}
			\sum_j  \norm{\abs{\nabla \chi_j^B}u}^2 \leq C R_0^{-2} B^{2\rho_1} \norm{u}^2_{L^2(\Omega)}.
		\end{equation}
		Moreover, \(\chi_1^B u\in \mathcal{D}(\mathcal{Q}_{\mathfrak{b}, \delta})\) and, combining Theorem \ref{Main Theorem} and min-max principle, we have
		\begin{equation}\label{Lower bound chi1}
			\mathcal{Q}_{\mathcal{A}, \Omega}(\chi_1^B u) \geq B \lambda_{\mathfrak{b}}(\delta) \norm{\chi_1 u}^2_{L^2(\Omega)}
		\end{equation}
		where we have used \eqref{Support chi_1^B} and the change of variables \(y= \sqrt{B}x\) (see \cite[Remark 3.3]{Bonnthesis}). For the elements in \(\sum_{\text{int}}\) we can adapt \cite[Lemma 1.4.1]{FournaisHelfferbook} as in \cite[Section 4]{AssaadBreakdown}. Note that any \(\chi_j^B\) inside of this sum has support contained in either \(\Omega_1\) or \(\Omega_2\). For \((x_1, x_2) \in \Omega_1\) we have the commutator relation
		\begin{equation*}
			B = i [\partial_{x_1} - i A_{\mathfrak{b}, \delta} , \partial_{x_2} ] 
		\end{equation*}
		and for \((x_1, x_2) \in \Omega_2\)
		\begin{equation*}
			B \abs{b} = -i [\partial_{x_1} - i A_{\mathfrak{b}, \delta} , \partial_{x_2} ].
		\end{equation*}
		Then one can prove, analogously as in \cite[Lemma 1.4.1]{FournaisHelfferbook}, that
		\begin{equation}\label{Lower bound interior}
			\sum_{\text{int}}  \mathcal{Q}_{\mathcal{A}, \Omega}(\chi_j^B u) \geq B \abs{b}  \sum_{\text{int}} \norm{\chi_j^B}^2_{L^2(\Omega)}. 
		\end{equation}
		For the last term, we can use the modal operator \(\mathcal{L}_{\mathfrak{b}, 0}\) because after a suitable rotation \(\chi_{j}^B u \in \mathcal{D}(\mathcal{Q}_{\mathfrak{b}, 0})\) by \eqref{Avoid origing condition} and \eqref{Two cases partition}. Thus, again by the min-max principle and \eqref{Beta a definition} 
		\begin{equation}\label{Lower bound gamma}
			\sum_{\text{bnd}}  \mathcal{Q}_{\mathcal{A}, \Omega}(\chi_j^B u) \geq B \beta_{\mathfrak{b}}  \sum_{\text{bnd}} \norm{\chi_j^B}^2_{L^2(\Omega)},
		\end{equation}
		where we have used the same change of variables as in \eqref{Lower bound chi1}. Let \(u\in H^1_0(\Omega)\) with \(\norm{u}_{L^2(\Omega)}=1\), then by min-max principle and combining \eqref{Last term lower bound}, \eqref{Lower bound chi1}, \eqref{Lower bound interior} and \eqref{Lower bound gamma} we obtain
		\begin{equation}\label{Lower bound final result}
			\lambda_1(B) \geq B\lambda_{\mathfrak{b}}(\delta)  - C R_0^{-2} B^{2\rho_1},
		\end{equation}
		which finishes the proof. 
	\end{proof}
	\subsection{Monotonicity of \(\lambda_1(B)\)}
	We want to study if \(B\mapsto \lambda_1(B)\) is monotone non-decreasing for some \(B>0\). This has been largely studied for Neumann magnetic Laplacian in smooth domains (see \cite{FournaisHelffer07strongdiamag,FournaisHelfferbook,FournaisSundqvistLackDiag}) and for magnetic steps in \cite{AssaadBreakdown}. The fact that we are avoiding Frenet coordinates and  considering a simpler setting will help us to simplify the technical part of the proofs. 
	
	We will mainly follow \cite[Section 8.5]{FournaisHelfferbook} and \cite[Section 6]{AssaadBreakdown}. Since \(\Omega\) is smooth, using a regularity theorem (one can adapt \cite[Section 9.6]{Brezis} for example) we have 
	\[\mathcal{D}(P^D_{B\mathcal{A}, \Omega}) = H^2(\Omega) \cap H^1_0(\Omega).\]
	An important observation is that the domain does not depend on \(B\), and this allows us to apply analytic perturbation theory (see \cite[Appendix C]{FournaisHelfferbook}). Moreover, the boundary conditions do not depend on \(B\). Thus, for any \(B>0\), we can define the left and right derivatives of \(\lambda_1 (B)\) as 
	\begin{equation}
		\lambda'_{1, \pm} (B) := \lim_{\epsilon \rightarrow \pm 0} \frac{\lambda_1(B+\epsilon) - \lambda_1(B)}{\epsilon}. 
	\end{equation}
	\begin{prop}\label{Normal Agmon estimates}
		Let \(u_1(\cdot; B)\) be a ground state of \(P^D_{B\mathcal{A}, \Omega}\) and \(B\geq 1\), then there exist constants \(\epsilon, R_0>0\) such that
		\begin{equation}
			\norm{u_1 e^ {\epsilon \abs{\cdot}}}^2_{L^2(\Omega)} + B^ {-1} \mathcal{Q}_{B\mathcal{A}, \Omega}(u_1 e^ {\epsilon \abs{\cdot}})\leq C_{R_0, \epsilon},
		\end{equation}
		where \(C_{R_0, \epsilon}\) is a positive constant. This implies that for all \(N>0\), 
		\begin{equation}\label{Polynomial Agmon estimates}
			\int_{\Omega}\abs{x}^N\abs{u_1(x; B)}^2 \, dx = \mathcal{O}(B^{-\frac{N}{2}}). 
		\end{equation}
	\end{prop}
	\begin{proof}
		This proof can be seen as an adaptation of \ref{Exponential decay eigenfunctions}. Let \(u_1\) be a ground state of \(P^D_{B\mathcal{A}, \Omega}\), and define for  \(\epsilon \in (0, \sqrt{\beta_{\mathfrak{b}} - \lambda_{\mathfrak{b}}(\delta)})\)
		\begin{equation*}
			g(x) = \epsilon \max\paren{\abs{x}, R_0 B^{-\rho_1}},
		\end{equation*}
		which is a real-valued Lipschitz function. Since \(\Omega\) is bounded
		\begin{equation*}
			e^{2\sqrt{B}g}u_1 \in L^2(\Omega) \quad \text{ and } \quad e^{\sqrt{B}g}u_1 \in \mathcal{D}(\mathcal{Q}_{B\mathcal{A}, \Omega}).
		\end{equation*}
		Using integration by parts and \(P^D_{B\mathcal{A}, \Omega} u_1 = \lambda_1(B) u_1\), we can proceed as in \eqref{Integration by parts Agmon} and \eqref{First bound quadratic form exp decay} to obtain 
		\begin{equation}\label{Thm 4 first equation proof}
			\mathcal{Q}_{B\mathcal{A}, \Omega}( e^{\sqrt{B}g}u_1)  = \lambda_1(B) \norm{e^{\sqrt{B}g}u_1}^2_{L^2(\Omega)} + B \norm{\abs{\nabla g}e^{\sqrt{B}g}}^2_{L^2(\Omega)} .
		\end{equation}
		Taking the same partition of unity as in Proposition \ref{Lower bound proposition}  and using IMS formula with \(e^{\sqrt{B}g}u_1\) we get
		\begin{equation*}
			\begin{split}
				\mathcal{Q}_{B\mathcal{A}, \Omega}( e^{\sqrt{B}g}u_1) \geq &  \paren{B\lambda_{\mathfrak{b}}(\delta) - CR_0^ {-2} B^{2\rho_1}}\int_{D(0, R_0 B^{-\rho_1})} \abs{e^{\sqrt{B}g}u_1}^ 2 \, dx\\
				& +\paren{B\beta_{\mathfrak{b}} - CR_0^ {-2} B^{2\rho_1}}\int_{\Omega\setminus D(0, R_0 B^{-\rho_1})} \abs{e^{\sqrt{B}g}u_1}^ 2 \, dx. 
			\end{split}
		\end{equation*}
		Combining this expression with \eqref{Thm 4 first equation proof} together with Proposition \ref{Proposition upper bound} and taking \(\rho_1 = \frac{1}{2}\)
		\begin{multline}\label{Mid step exp decay monotonicity}
			\paren{\lambda_{\mathfrak{b}}(\delta) - CR_0^ {-2}}\int_{D(0, R_0 B^{-\frac{1}{2}})} \abs{e^{\sqrt{B}g}u_1}^ 2 \, dx \\
			+\paren{\beta_{\mathfrak{b}} - CR_0^{-2} }\int_{\Omega\setminus D(0, R_0 B^{-\frac{1}{2}})} \abs{e^{\sqrt{B}g}u_1}^ 2 \, dx \\
			\leq \paren{\lambda_{\mathfrak{b}}(\delta) +C_2\underbrace{B^ {-\frac{1}{2}}}_{=o(1)}}  \norm{e^{\sqrt{B}g}u_1}^2_{L^2(\Omega)} + \norm{\abs{\nabla g}e^{\sqrt{B}g}}^2_{L^2(\Omega)} .
		\end{multline}
		By construction, \(\abs{\nabla g}\leq \epsilon\) and \(\text{supp }(\nabla g) \subset D(0, R_0B^{-\frac{1}{2}})\), which means
		\begin{equation*}
			\norm{\abs{\nabla g}e^{\sqrt{B}g}}^2_{L^2(\Omega)}  \leq \epsilon \int_{ D(0, R_0B^{-\frac{1}{2}})} \abs{e^{\sqrt{B}g}u_1}^ 2 \, dx .
		\end{equation*}
		Plugging this into \eqref{Mid step exp decay monotonicity} and simplifying terms we get
		\begin{multline*}
			\paren{\beta_{\mathfrak{b}}- \lambda_{\mathfrak{b}}(\delta) - \epsilon^2 - CR_0^{-2} }\int_{\Omega\setminus D(0, R_0 B^{-\frac{1}{2}})} \abs{e^{\sqrt{B}g}u_1}^ 2 \, dx \\
			\leq \paren{CR_0^ {-2} +C_2\underbrace{B^ {-\frac{1}{2} }}_{=o(1)}} \int_{ D(0, R_0B^{-\frac{1}{2}})} \abs{e^{\sqrt{B}g}u_1}^ 2 \, dx \\
			\leq \paren{CR_0^ {-2} +C_2\underbrace{B^ {-\frac{1}{2} }}_{=o(1)}} e^{2R_0} \int_{ D(0, R_0B^{-\frac{1}{2}})} \abs{u_1}^ 2 \, dx, 
		\end{multline*}
		then for \(R_0>0\) large enough we have that \(\paren{\beta_{\mathfrak{b}}- \lambda_{\mathfrak{b}}(\delta) - \epsilon^2 - CR_0^{-2} }>0\) and this means that we can find a constant \(C_{R_0, \epsilon}\) such that 
		\begin{equation}\label{Monot exp decay function}
			\norm{u_1 e^ {\epsilon\abs{\cdot}}}^2_{L^2(\Omega)} \leq C_{R_0, \epsilon},
		\end{equation}
		since \(\epsilon \abs{x}\leq \epsilon g(x)\). Similarly, we can combine \eqref{Thm 4 first equation proof} and \eqref{Monot exp decay function},
		\begin{equation}
			\norm{u_1 e^ {\epsilon\abs{\cdot}}}^2_{L^2(\Omega)} + B^ {-1} \mathcal{Q}_{B\mathcal{A}, \Omega}(u_1 e^ {\epsilon\abs{\cdot}})\leq C'_{R_0, \epsilon},
		\end{equation}
		for some constant \(C'_{R_0, \epsilon}>0\). 
	\end{proof}
	\begin{theorem}\label{Monotonicity theorem}
		Let \(\mathfrak{b} =  (1, b)\) with \(b\in (-1,0)\) and \(\delta_0 \in (0, \frac{\pi}{2})\) such that Theorem \ref{Main Theorem} holds for \(\mathcal{L}_{\mathfrak{b}, \delta}\) with \(\delta < \delta_0\). We have
		\begin{equation}\label{Monotonicity equation}
			\lim_{B\rightarrow +\infty} \lambda'_{1,+}(B) = \lim_{B\rightarrow +\infty} \lambda'_{1,-} (B)= \lambda_{\mathfrak{b}}(\delta).
		\end{equation}
		In particular \(B\mapsto \lambda_1(B)\) is strictly increasing for large \(B\). 
	\end{theorem}
	\begin{proof}
		By analytic perturbation theory we know that there exist an \(\epsilon > 0\) and analytic functions
		\begin{equation*}
			(B-\epsilon, B+\epsilon)\ni \theta \mapsto u_j(\cdot \,  ; \theta)\in H^2(\Omega)\cap H_0^1(\Omega)\setminus\{0\},
		\end{equation*}
		\begin{equation*}
			(B-\epsilon, B+\epsilon)\ni \theta \mapsto E_j \in \R, 
		\end{equation*}    
		for \(j\in \{1, \ldots, n\}\) and some \(n\in \N\). For \(\theta >B\) let \(E_{j_+}(\theta) = \min_{j\in \{1, \ldots, n\}} E_j(\theta)\) and, for \(\theta<B\), let \(E_{j_-}(\theta) = \min_{j\in \{1, \ldots, n\}} E_j(\theta)\). Observe that \(\lambda'_{1,+}(B) = E'_{j_+}(B)\) and \(\lambda'_{1,-}(B) = E'_{j_-}(B)\). By first order perturbation theory (see \cite[Proposition 2.3.1]{FournaisHelfferbook}) we know
		\begin{equation}\label{Derivative = 2 Im}
			\lambda'_{1,\pm}(B) =2\Im\langle u_{j_{\pm}}, \mathcal{A}\cdot (\nabla -i B\mathcal{A})u_{j_{\pm}}\rangle_{L^ 2(\Omega)}.
		\end{equation}
		Since 
		\begin{equation*}
			2 \mathcal{A}\cdot (\nabla -i B\mathcal{A})= \frac{-i}{\epsilon} \big (P^D_{(B+\epsilon)\mathcal{A}, \Omega} -P^D_{B\mathcal{A}, \Omega} \big ) + i \epsilon \abs{\mathcal{A}}^2 - \nabla \cdot \mathcal{A}, 
		\end{equation*}
		we get, for \(B>0\),
		\begin{equation}
			\lambda'_{1,+}(B) \geq \frac{\lambda_1(B+\epsilon)- \lambda_1(B)}{\epsilon} - \epsilon \int_{\Omega} \abs{\mathcal{A}}^2 \abs{u_{j_+}}^2 \, dx,
		\end{equation}
		where \(\epsilon>0\). Recall that \(\mathcal{A}= (\mathcal{A}_{\mathfrak{b}, \delta},0)\) where \(\mathcal{A}_{\mathfrak{b}, \delta}\) was defined in \eqref{Aa,b,delta definition}, thus there exists a constant \(C_{\mathfrak{b}, \delta}>0\) such that \(\abs{\mathcal{A}(x)}\leq C_{\mathfrak{b}, \delta} \abs{x}\) for all \(x\in \Omega\). Then, by Proposition \ref{Normal Agmon estimates} we have that
		\begin{equation}
			\int_{\Omega} \abs{\mathcal{A}}^2 \abs{u_{j_+}}^2 \, dx \leq \tilde{C}_{\mathfrak{b}, \delta} B^{-1},
		\end{equation}
		for some constant \( \tilde{C}_{\mathfrak{b}, \delta}>0\). Take \(\epsilon = \eta B\) for some \(\eta >0\). Using the asymptotics of \(\lambda_1(B)\) established in Theorem \ref{Theorem 2}, we obtain
		\begin{equation}
			\liminf_{B\rightarrow +\infty} \lambda'_{1,+}(B) \geq \lambda_{\mathfrak{b}}(\delta)- \eta \tilde{C}_{\mathfrak{b}, \delta}. 
		\end{equation}
		Since the choice of \(\eta \) was arbitrary we have
		\begin{equation}\label{Monotonicity liminf}
			\liminf_{B\rightarrow +\infty} \lambda'_{1,+}(B) \geq \lambda_{\mathfrak{b}}(\delta). 
		\end{equation}
		Similarly, for \(\epsilon <0\) we can start from \eqref{Derivative = 2 Im} to get to
		\begin{equation}\label{Monotonicity limsup}
			\limsup_{B\rightarrow +\infty} \lambda'_{1,-}(B) \leq \lambda_{\mathfrak{b}}(\delta),
		\end{equation}
		where the inequalities are reversed due to the sign of \(\epsilon\). By perturbation theory \(\lambda'_{1,+}(B)\leq \lambda'_{1,-}(B)\). Hence, combining this with \eqref{Monotonicity liminf} and \eqref{Monotonicity limsup} we finish the proof. 
	\end{proof}
	\appendix
	\section{Bottom of the essential spectrum}\label{AppendixA}
	This section follows \cite[Section 3 and Appendix A]{AssaadBreakdown} to give a version of Persson's lemma \cite{Persson'sLemma} for \(\mathcal{L}_{\mathfrak{b},  \delta}\). This known result indicates that the essential spectrum of a self-adjoint operator is not affected by what happens in compact sets. Note that \(\La\) is self-adjoint, positive and satisfies the IMS--Formula. We will omit most of the details since the proofs are similar to the ones that can be found in \cite{AssaadBreakdown}. Before starting with the computation of the bottom of the essential spectrum of \(\La\), it is useful to introduce some notation. 
	\begin{equation}\label{M_R}
		\mathcal{M}_R := \{u\in \mathcal{D}(\mathcal{L}_{\mathfrak{b},  \delta}) : \text{supp }u \subset \{|x|>R\}\}, 
	\end{equation}
	\begin{equation}\label{Sigma_R}
		\Sigma_R := \inf \{\langle u , \mathcal{L}_{\mathfrak{b},  \delta}u\rangle : u \in \mathcal{M}_R, \norm{u}_{L^2(\R^2)}=1\},
	\end{equation}
	and we also define 
	\begin{equation}\label{1.8}
		\Sigma := \lim_{r\rightarrow +\infty}\Sigma_R = \sup_{r>0}\Sigma_R ,
	\end{equation}
	which is called ionization threshold when we deal with Schr\"odinger operators for atoms and molecules. Note that \(\Sigma_R\) is an increasing function since \(\mathcal{M}_{R_2}\subset \mathcal{M}_{R_1}\) if \(R_1 < R_2\), this means that \(\Sigma \in (0, +\infty]\) because \(\mathcal{L}_{\mathfrak{b},\delta}\) is positive.
	\begin{defin}[Weyl sequence]\label{Weyl sequence}
		\((u_n)_{n} \subset \mathcal{D}(\mathcal{L}_{\mathfrak{b},  \delta})\) is called a \textbf{Weyl sequence} of \(\lambda \in \R\) and \(\mathcal{L}_{\mathfrak{b},  \delta}\) if \(\norm{u_n}_{\mathcal{H}} =1 \) ,  \(u_n \rightharpoonup 0 \) and \(\norm{(\mathcal{L}_{\mathfrak{b},  \delta}-\lambda)u_n}_{L^2(\R^2)}\rightarrow 0\). 
	\end{defin}
	By Weyl's criterion \cite[Section 25.10]{GustafsonSigal}, we know that 
	\begin{equation*}
		\text{Spec}_{\text{ess}}(\La) = \{\lambda\in \R \text{ such that there exists a Weyl sequence of }\lambda \text{ and } \La\}.
	\end{equation*} 
	\begin{lemma}\label{Lemma 1 Appendix}
		Let \((u_n)_{n} \) be a Weyl sequence of \(\lambda \in \R\) and the operator \(\mathcal{L}_{\mathfrak{b},  \delta}\). Then for all \(R>0\), we have that \(\lim_{n\rightarrow + \infty}\norm{u_n}_{L^2(D(0,R))}=0\). 
	\end{lemma}
	\begin{proof}
		We omit this proof. We refer to\cite[Lemma A.3]{AssaadBreakdown} for details of a similar proof. 
	\end{proof}
	\begin{theorem}\label{Theorem Persson's Lemma}
		The essential spectrum of \(\La\) satisfies
		\begin{equation*}
			\inf\text{Spec}_{\text{ess}} (\La) = \Sigma,
		\end{equation*}
		where \(\Sigma \) was defined in \eqref{1.8}. 
	\end{theorem}
	\begin{proof}
		We omit this proof. We refer to \cite[Theorem A.1]{AssaadBreakdown} for details of a similar proof. 
	\end{proof}
	\begin{lemma}\label{Lemma appendix}
		Let \(\mathfrak{b} = (1, b)\) or \(\mathfrak{b} = (b,1)\) with \(b\in [-1,0)\), then there exists a constant \(C>0\) independent of \(b\), such that for all \(R>0\) and any non zero function \(u\in \mathcal{M}_R\) we have
		\begin{equation}\label{Bound C/R^2}
			\mathcal{Q}_{\mathfrak{b},\delta}(u) \geq \bigg (\beta_{\mathfrak{b}} -\frac{C}{R^2}\bigg) \norm{u}^2_{L^2(\R^2)}
		\end{equation}
	\end{lemma}
	\begin{proof}
		The main idea is introducing a partition of unity to use the information we know about the operators introduced in Section \ref{Introduction section} together with IMS-Formula. We introduce a partition of unity \((\hat{\chi}_j)\) for the interval \([0,2\pi]\) such that\footnote{The partition of unity could be adapted and make it dependent on \(\delta\), but for this particular case there is no need.}
		\begin{equation*}
			\supp \hat{\chi}_1 \subset \bigg[0, \frac{2\pi}{3} \bigg]\cup \bigg[\frac{4\pi}{3}, 2\pi\bigg] \text{ and }  \supp \hat{\chi}_2 \subset  \bigg[\frac{\pi}{3}, \frac{5\pi}{3}\bigg], 
		\end{equation*}
		with \(\abs{ \hat{\chi}_1 (\theta)}^2 + \abs{ \hat{\chi}_2(\theta)}^2 =1\) and \(\abs{ \hat{\chi}'_1 (\theta)}^2 + \abs{ \hat{\chi}'_2(\theta)}^2 \leq C\) for all \(\theta\in [0, 2\pi]\) where \(C>0\) is some positive constant. We define the partition of unity in polar coordinates as
		\begin{equation*}
			\chi_j (r, \theta) := \hat{\chi}_j (\theta) \qquad j\in \{1,2\}, 
		\end{equation*}
		for all \( (r, \theta)\in (0, +\infty) \times (0,2\pi)\). Denote by
		\begin{equation*}
			\Tilde{\chi}_j(x_1, x_2) = \chi_j(r, \theta) \qquad j\in \{1,2\}
		\end{equation*}
		the partition of unity considered with cartesian coordinates. Let \(\varphi\in \mathcal{M}_R\setminus\{0\}\). The IMS localization formula asserts that 
		\begin{equation}\label{IMS formula}
			\norm{(\nabla -i \mathcal{A}_{\mathfrak{b}, \delta}) \varphi}^2_{L^2(\R^2)} = \sum_{j=1}^2 \norm{(\nabla -i \mathcal{A}_{\mathfrak{b}, \delta})(\Tilde{\chi}_j\varphi)}^2_{L^2(\R^2)} - \sum_{j=1}^2 \norm{\varphi\abs{\nabla \Tilde{\chi}_j}}^2_{L^2(\R^2)}.
		\end{equation}
		For \((x_1, x_2) \in \R^2\) and \(j\in \{1,2\}\)
		\begin{equation*}
			\abs{\nabla \Tilde{\chi}_j(x_1, x_2) } = \underbrace{\abs{\partial_r \chi_j(r, \theta)}}_{=0} + \frac{1}{r^2} \abs{\partial_{\theta}  \chi_j(r, \theta)}^2 = \frac{1}{r^2} \abs{\partial_{\theta}  \chi_j(r, \theta)}^2. 
		\end{equation*}
		Then, by construction of \(\chi_j(r,\theta)\) and since \( \text{supp }\varphi \subset \{|x|>R\}\), we have 
		\begin{equation}\label{Second term IMS}
			\sum_{j=1}^2 \norm{\varphi\abs{\nabla \Tilde{\chi}_j}}^2_{L^2(\R^2)}\leq \frac{C}{R^2} \norm{\varphi}^2_{L^2(\R^2)}.
		\end{equation}
		For the other term in \eqref{IMS formula} note that extending \(\Tilde{\chi}_j \varphi\) by zero gives a function in the domain of \(\mathcal{Q}_{\mathfrak{b}, 0}\), where this quadratic form was defined \eqref{Q_a,b,delta}. In other words, we can connect \(\Tilde{\chi}_j \varphi\) with the situation of a magnetic step in the whole plane given by \(\mathcal{L}_{\mathfrak{b}, 0}\) (we are using again the rotation invariance of this operator). Hence, for \(j\in\{1,2\}\) 
		\begin{equation}\label{First term IMS}
			\frac{\norm{(\nabla -i \mathcal{A}_{\mathfrak{b}, \delta})(\Tilde{\chi}_j\varphi)}^2_{L^2(\R^2)}}{\norm{\Tilde{\chi}_j\varphi}^2_{L^2(\R^2)}} \geq  \inf_{u\in \mathcal{D}(\mathcal{Q}_{\mathfrak{b},0}) \setminus \{0\}} \frac{\norm{(\nabla -i \mathcal{A}_{\mathfrak{b}, 0})u}^2_{L^2(\R^2)} }{\norm{u}^2_{L^2(\R^2)}} = \beta_{\mathfrak{b}}. 
		\end{equation}
		Combining \eqref{IMS formula}, \eqref{Second term IMS} and \eqref{First term IMS} we obtain
		\begin{equation}
			\mathcal{Q}_{\mathfrak{b}, \delta}(\varphi) \geq\bigg(\beta_{\mathfrak{b}} - \frac{C}{R^2 }\bigg)\norm{\varphi}^2_{L^2(\R^2)}. 
		\end{equation}
	\end{proof}
	\begin{theorem}\label{Essential spectrum Theorem}
		Let \(\mathfrak{b}=(1, b)\) or \(\mathfrak{b}=( b, 1)\) with \(b\in (-1,0)\), then \[\inf\text{Spec}(\mathcal{L}_{\mathfrak{b}, \delta})=\beta_{\mathfrak{b}}.\]
	\end{theorem}
	\begin{proof}
		By Theorem \ref{Theorem Persson's Lemma}, it is enough to show that \(\Sigma = \beta_{\mathfrak{b}}\). 
		\begin{itemize}
			\item \( \beta_{\mathfrak{b}} \leq \Sigma\). By \eqref{Sigma_R} and Lemma \ref{Lemma appendix} we know that for any \(R>0\)
			\begin{equation*}
				\Sigma_R \geq \beta_{\mathfrak{b}} - \frac{C}{R^2}. 
			\end{equation*}
			Taking \(R\rightarrow +\infty\) gives the desired inequality.
			\item \(\beta_{\mathfrak{b}} \geq \Sigma\). Let \(\epsilon, R >0\), in \eqref{Beta a definition} we saw that \(\beta_{\mathfrak{b}}\) is the bottom of the spectrum of \(\mathcal{L}_{\mathfrak{b}, 0}\). Then, combining min-max principle with a standard limiting argument we can find an \(R>0\) and function \(u\in \mathcal{D}(\mathcal{Q}_{a,b})\) with \(\text{supp }u \subset D(0,R)\) such that 
			\begin{equation*}
				\beta_{\mathfrak{b}}\leq \frac{\mathcal{Q}_{\mathfrak{b}, 0}(u)}{\norm{u}_{L^2(\R^2)}^2}\leq \beta_{\mathfrak{b}} + \epsilon . 
			\end{equation*}
			Notice that \(\mathcal{A}_{\mathfrak{b}, \delta}(x) =\mathcal{A}_{\mathfrak{b}, 0}(x) \) for all \(x\in \R_+ \times \R\). Then we can define \(v(x_1, x_2) := u(x_1+2R, x_2)\) such that \(v\)  is supported in \((\R_+ \times \R) \cap (\R^2 \setminus D(0, R))\) and 
			\begin{equation*}
				\frac{\mathcal{Q}_{\mathfrak{b}, 0}(u)}{\norm{u}^2_{L^2(\R^2)}}=\frac{\mathcal{Q}_{\mathfrak{b},\delta}(v)}{\norm{v}^2_{L^2(\R^2)}}. 
			\end{equation*}
			Hence, 
			\begin{equation*}
				\Sigma_R \leq \frac{\mathcal{Q}_{\mathfrak{b},\delta}(v)}{\norm{v}^2_{L^2(\R^2)}} \leq \beta_{\mathfrak{b}} + \epsilon.
			\end{equation*}
			Taking \(\epsilon\) to zero and \(R\) to \(+\infty\) gives the result. 
		\end{itemize}
	\end{proof}
	\section{Decay properties of \(\phi_{\mathfrak{b}}\)}\label{AppendixB}
	To prove Theorem \ref{Main Theorem} we used decay properties of \(\phi_{\mathfrak{b}}\) and its derivative. To show this, we will follow \cite[Section 2.4]{Raymondbook} adapted to our case. To this end we need to introduce \(U(a,t)\) as the first Weber parabolic function which is a solution of the parabolic cylindric equation
	\begin{equation}\label{Weber function ODE}
		-u''(t) + \frac{1}{4}t^2 u(t) = -au(t) , 
	\end{equation}
	and let \(\hat{U}:= \Re (U)\). Note \eqref{Weber function ODE} has an irregular singular point at infinity, and an standard ODE theory\footnote{One can use the transformation \(x=1/t\) to study the singular point at zero and, after a Liouville transformation, use \cite[Theorem X.10]{ReedSimon2}. Using exponential substitution as in \cite[Chapter 3]{BenderOrszag} can give a first intuition of the eigenfunction asymptotics.} guarantees that \(\hat{U}(a,t)\) is the only solution in \(L^2\).  
	
	Since we are dealing with ``trapping magnetic steps'' (see Remark \ref{Remark b cases}), we can assume that \(b_2 =1 \). Let 
	\begin{equation}\label{Weber positive function}
		\phi_{\mathfrak{b}} (t) := \begin{cases}
			C \hat{U}\paren*{-\frac{\beta_{\mathfrak{b}}}{2} , \sqrt{2}(t+\xi_{\mathfrak{b}})} &  \text{ for } t\geq 0 \\
			C \hat{U}\paren*{-\frac{\beta_{\mathfrak{b}}}{2\abs{b_1}} , \sqrt{\frac{2}{|b_1|}}(b_1 t+\xi_{\mathfrak{b}})}  & \text{ for } t< 0
		\end{cases},
	\end{equation}
	where \(C>0\) is some normalization constant. We have that \(\mathfrak{h}_{\mathfrak{b}}[\xi_{\mathfrak{b}}]\phi_{\mathfrak{b}} = \beta_{\mathfrak{b}} \phi_{\mathfrak{b}}\) and \(\phi_{\mathfrak{b}}\) decays exponentially as \(t\rightarrow \pm \infty\). Moreover, we get that both \(\phi_{\mathfrak{b}}\) and \(\phi'_{\mathfrak{b}}\) are real-valued. 
	
	\section*{Acknowledgements}
	I would like to thank A. Kachmar for the suggestion and discussion of the problem, and M. Persson Sundqvist for his attentive reading and useful suggestions.
	\bibliographystyle{plain} 
	\bibliography{mybib} 
\end{document}